\documentclass[12pt]{amsart}
\usepackage{amscd}
\usepackage{amsfonts,amssymb,amsthm}
\usepackage[numbers]{natbib}
\usepackage{enumerate}
\usepackage{graphicx}
\usepackage{caption}
\usepackage{subcaption}

\theoremstyle{plain}
\newtheorem{theorem}{Theorem}[section]
\newtheorem{lemma}[theorem]{Lemma}
\newtheorem{corollary}[theorem]{Corollary}
\newtheorem{proposition}[theorem]{Proposition}

\theoremstyle{definition}
\newtheorem*{definition}{Definition}

\newtheorem*{theorem*}{Theorem}
\theoremstyle{remark}
\newtheorem{remark}{Remark}

\newcommand{\aca}{\'}   
\newcommand{\acg}{\`}   
\newcommand{\virgolap}{``}  
\newcommand{\virgolch}{''}  

\DeclareMathOperator{\inj}{\inj}

\begin{document}
\title{Sphere systems, standard form, and cores of products of trees}
\author{Francesca Iezzi}
\address{Mathematics Institute, University of Warwick,
Coventry, CV4 7AL, Great Britain}
\date{25th October 2016}

\begin{abstract}
We introduce the concept of a \emph{standard form} for two embedded maximal sphere systems in the doubled handlebody, and we prove an existence and uniqueness result. In particular, we show that pairs of maximal sphere systems in the doubled handlebody (up to homeomorphism) bijectively correspond to square complexes satisfying a set of properties. This work is a variant on Hatcher's normal form.
\end{abstract}

\maketitle

\section{Introduction}\label{SA}

Let $M_g$ be the connected sum of $g$ copies of $S^2 \times S^1$, this is homeomorphic to the double of the handlebody of genus $g$. Note that the fundamental group of $M_g$ is the free group of rank $g$, denoted as $F_g$, and, if $Mod(M_g)$ denotes the group of isotopy classes self-homeomorphisms of $M_g$, the natural map $Mod(M_g) \rightarrow Out(F_g)$ 
is surjective with finitely generated kernel; moreover, elements of the kernel fix homotopy classes of spheres (as proven in \citep{Lau2} p. 80-81). For this reason, collections of spheres in this class of manifolds have been a significant tool in the study of outer automorphism groups of free groups. We refer to a collection of disjoint pairwise non isotopic spheres in $M_g$ as a \emph{sphere system}.
An important result is that homotopic sphere systems in $M_g$ are isotopic (\citep{Lau1} Th\aca{e}or\acg{e}me I).


The idea of using sphere systems in $M_g$ as a tool in the study of $Out(F_g)$ goes back to Whitehead (\citep{Whi}, \citep{Sta})
and has been further developed by Hatcher in \citep{Hat1}. In the latter the author introduces the \emph{sphere complex} of the manifold $M_g$,
which has been a very useful tool in the study of the groups $Out(F_g)$. Collections of spheres in $M_g$ can also be used to give definitions for the free factor complex
(\citep{HatVog2}), and for Culler Vogtmann Outer Space (Appendix of  \citep{Hat1}).

In \citep{Hat1} the author also introduces the concept of a \emph{normal form} of spheres  with respect to a given maximal sphere system and he proves an existence result. Hatcher's normal form has been the basis of many of the results concerning the sphere complex, for example, the proof of hyperbolicity (\citep{HilHor}).



\vskip 0.3cm
In this paper, we introduce the concept of a \emph{standard form} for a pair of maximal sphere systems $(\Sigma_1, \Sigma_2)$ in $M_g$. Standard form is a symmetric definition, and is equivalent to reciprocal normal form of $\Sigma_1$, $\Sigma_2$, with the additional requirement that all complementary components of $\Sigma_1 \cup \Sigma_2$ in $M_g$ are handlebodies. We show then an existence and uniqueness result:

\begin{theorem*}I
Given a pair of maximal sphere systems $(\Sigma_1, \Sigma_2)$ in $M_g$ there exists a homotopic pair $(\Sigma_1', \Sigma_2')$ in standard form.
\end{theorem*}

\begin{theorem*}II
If $(\Sigma_1, \Sigma_2)$ and  $(\Sigma_1', \Sigma_2')$ are two homotopic pairs of maximal sphere systems in $M_g$ in standard form, then there is a homeomorphism  $F:M_g \rightarrow M_g$, which induces an inner automorphism of the fundamental group and so that $F$ maps the pair $(\Sigma_1, \Sigma_2)$ to the pair $(\Sigma_1', \Sigma_2')$.
\end{theorem*}

Note that Theorem I and Laudenbach's result that homotopic spheres in $M_g$ are isotopic imply that any pair of maximal sphere systems in $M_g$ can be isotoped to be in standard form.

\vskip 0.2cm
The methods we use are combinatorial. The main idea is that, given a pair of maximal sphere systems $(\Sigma_1, \Sigma_2)$, one can build a dual square complex (i.e. a CAT(0) cube complex of dimension $2$), where $0$-cells correspond to complementary components of $\Sigma_1 \cup \Sigma_2$, 1-cells correspond to components of $\Sigma_1 \setminus \Sigma_2$ and of $\Sigma_2 \setminus \Sigma_1$, and 2-cells correspond to components of $\Sigma_1 \cap \Sigma_2$.
The same square complex can also be obtained by applying the construction described in \citep{Gui} to the dual trees to $\widetilde{\Sigma_1}$ and $\widetilde{\Sigma_2}$ in the universal cover $\widetilde{M_g}$.
This idea has been used in \citep{Hor} to estimate distances in Outer Space.

In this paper (Section \ref{Inverse construction}) we describe an inverse to this construction, i. e. we show that
square complexes endowed with a set of properties determine pairs of sphere systems in $M_g$ in standard form.



\vskip 0.3cm
\par The article is organised as follows.

In Section \ref{Spheres, partitions and intersections} we clarify notation and
we recall how spheres in $M_g$ and intersection numbers relate to partitions of the space of ends of the universal cover.
\par In Section \ref{Standard form} we introduce the definition of
 \emph{standard form} for a pair of maximal sphere systems $(\Sigma_1, \Sigma_2)$ and
we hint at how standard form of $(\Sigma_1, \Sigma_2)$ implies some properties of the dual square complex. 

In Section \ref{The core of two trees} we introduce a more abstract construction:
given two trivalent trees $T_1$, $T_2$ endowed with a boundary identification we construct a \emph{core} $C(T_1, T_2)$ and we show that the core satisfies a set of properties  (properties (1)-(5) on page \pageref{properties core}).
The core  $C(T_1, T_2)$ we define turns out to be the same as the \emph{Guirardel core} of $T_1$ and $T_2$, defined in \citep{Gui}; .
We give a combinatorial description of this object, using partitions of the space of ends.
The construction of Section \ref{The core of two trees} gives an alternative way of building the dual square complex to two sphere systems in standard form (as shown in Proposition 2.1 of \citep{Hor} and in Proposition \ref{equalityconstructions} below).

In Section \ref{Inverse construction} we show that, starting with a square complex endowed with properties (1)-(5) of Section \ref{The core of two trees}, we can construct a doubled handlebody, with two embedded maximal sphere systems in standard form.

As an application of the constructions of Section \ref{The core of two trees} and Section \ref{Inverse construction},
in Section \ref{Consequences}
we prove Theorem I and Theorem II.
\vskip 0.3cm

Throughout the paper, for the sake of simplicity, we always assume that a pair of maximal sphere systems $(\Sigma_1, \Sigma_2)$ in $M_g$ satisfies the following hypothesis:
\par ($*$) no sphere in
$\Sigma_1$ is homotopic to any sphere in $\Sigma_2$
\label{hypothesis ($*$)}
\par All the arguments of the article can be generalised to the case where ($*$) is not fulfilled. A discussion about this more general case and a hint on how to generalise the arguments can be found in Section 2.6 of \citep{Iez}.
\vskip 0.2cm
 Note that Theorem I could also be proved using Hatcher's existence theorem for normal form. However, our arguments are independent on Hatcher's work, even though the concept of normal form has served as an inspiration.
\vskip 0.3cm
Acknowledgments: This work was carried out during my PhD, under the supervision of Brian Bowditch. I am very grateful for his guidance. My PhD was funded by an EPSRC doctoral grant. I have written this article while supported by Warwick Institute of Advanced Study (IAS) and Warwick Institute of Advanced Teaching and Learning (IATL).

\section{Spheres, partitions and intersections}
\label{Spheres, partitions and intersections}

Throughout the paper, we denote by $M_g$ the connected sum of $g$ copies of $S^2 \times S^1$ (i. e. the doubled handlebody), and by $\widetilde{M_g}$ its universal cover.
We always suppose spheres in $M_g$ are embedded and intersect transversally. A sphere is \emph{essential} if it does not bound a ball.
We denote by $i(s_1, s_2)$ the minimum possible number of circles belonging to $s_1 \cap s_2$, over the homotopy class of $s_1$ and $s_2$ and we call this number the intersection number of the spheres $s_1$ and $s_2$.   We say that two spheres $s_1$, $s_2$ intersect minimally if they realise their intersection number.
A \emph{sphere system} in $M_g$ is a collection of non isotopic disjoint spheres. We call a sphere system $\Sigma$ \emph{maximal} if it is maximal under inclusion, i.e. any embedded essential sphere $\sigma$ is either isotopic to a component of $\Sigma$, or intersects $\Sigma$.
A maximal sphere system $\Sigma$ in $M_g$ contains $3g-3$ spheres, and all connected components of $M_g \setminus \Sigma$ are three holed spheres.
\vskip 0.2cm
This section contains some (already known) background results about embedded spheres in the manifold $M_g$.
In particular, we will recall that embedded spheres in $\widetilde{M_g}$ can be identified to partitions of the space of ends of $\widetilde{M_g}$, which can be identified to the boundary of a given tree (Lemma \ref{spherepartition}); furthermore, the intersection number of two spheres in $\widetilde{M_g}$ is positive if and only if the partitions associated to the two spheres satisfy a particular property (Lemma \ref{sphereintersection}).

We first recall that, given a sphere system $\Sigma$ in $M_g$ (or in $\widetilde{M_g}$) we can associate to $\Sigma$ a graph $G_{\Sigma}$. Namely we take a vertex $v_C$ for each component $C$ of $M_g \setminus \Sigma$ and an edge $e_{\sigma}$ for each sphere $\sigma$ in $\Sigma$. The edge $e_{\sigma}$ is incident to the vertex $v_C$ if the sphere $\sigma$ is one of the boundary components of $C$. We call $G_{\Sigma}$ the \emph{dual graph to $\Sigma$}. We can endow $G_{\Sigma}$ with a metric by giving each edge length one.
There is a natural retraction $r: M_g \rightarrow G_\Sigma$. Namely, consider a regular neighborhood of $\Sigma$, call it $U(\Sigma)$ and parameterise it as $\Sigma \times (0,1)$. For any component $C$ of $M_g \setminus U(\Sigma$) let $r|C$ map everything to the vertex $v_C$. For any sphere $\sigma$ in $\Sigma$ set $r(\sigma \times t)$ to be the point $t$ in $e_{\sigma}$.
If each complementary component of $\Sigma$ in $M_g$ is simply connected, then the retraction $r$ induces an isomorphism of fundamental groups. Note that if $\widetilde{\Sigma}$ is the full lift of $\Sigma$ in $\widetilde{M_g}$, then the dual graph to $\widetilde{\Sigma}$ in $\widetilde{M_g}$ (which is a tree) is isomorphic to the universal cover of the graph $G_\Sigma$. We will denote it as $T_\Sigma$ and call it the \emph{dual tree} to $M_g$ and $\Sigma$. Note also that the retraction $r:M_g \rightarrow G_\Sigma$ lifts to a retraction $h: \widetilde{M_g} \rightarrow T_\Sigma$. If $\Sigma$ is a maximal sphere system, then the dual graph $G_\Sigma$ is trivalent, as well as the dual tree $T_\Sigma$.



\vskip 0.2cm

\subsection{Space of ends} Next we recall the definition of the space of ends.
If $X$ is a topological space and  $\{K_n\}$ is an exhaustion of $X$ by compact sets, then an \emph{end} of $X$ is a sequence $\{U_n\}$ where $U_k$ is a component of $X\setminus K_k$ and $U_k \supset U_{k+1}$.
This definition does not depend on the particular sequence of compact sets we choose. Given an open set $A$ in $X$ we say that an end $\{U_n\}$ is contained in the set $A$ if, for $k$ large enough, the set $U_k$ is contained in $A$.
Call the collection of ends of a given space $X$ the \emph{space of ends} of $X$ and denote it by $End(X)$.
The space $End(X)$ can be endowed with a topology:
a fundamental system of neighborhoods for the end $\{U_n\}$ is
given by the sets $\{e_{U_k} \}$,
for $U_k \in \{U_n \}$,
 where $e_{U_k}$ consists of all the points in $End(X)$ contained in $U_k$.
Note that the space of ends of a tree can be identified with its Gromov boundary, which is a Cantor set.
We refer to chapter 8 of \citep{BriHae} and to \citep{Pes} for some more detailed background.
\vskip 0.2cm

We observe next that the space of ends of the manifold  $\widetilde{M_g}$  can be identified to the space of ends of a given tree.
In fact, consider $M_g$ with an embedded maximal sphere system $\Sigma$ and let $T_\Sigma$ be
the dual tree to $M_g$ and $\Sigma$ (note that $T_\Sigma$ is trivalent since by maximality of $\Sigma$).
The retraction $h: \widetilde{M_g} \rightarrow T_\Sigma$ induces a homeomorphism between the space $End(T_\Sigma)$ and the space $End(\widetilde{M_g})$. The latter is therefore a Cantor set.
\vskip 0.2cm
\par Now, since $\widetilde{M_g}$ is simply connected, every sphere $\sigma \subset \widetilde{M_g}$ separates, and induces a partition of the space of ends. The partition induced by a sphere $\sigma$ on the space of ends of $\widetilde{M_g}$ coincides with the partition induced by the corresponding edge $e_\sigma$ on the boundary of the dual tree $T_\Sigma$.
\vskip 0.2cm
\par To fix terminology, if $C$ is a Cantor set and $P_1 \doteq C=(P_1  ^+ \cup P_1  ^-)$, $P_2 \doteq C= (P_2  ^+ \cup P_2  ^-)$ are two distinct partitions of the set $C$, we say that $P_1$ and $P_2$ are \emph{non-nested} if all four sets $P_1^+ \cap  P_2^+$, $P_1^+ \cap  P_2^-$, $P_1^- \cap  P_2^+$, $P_1^- \cap  P_2^-$ are non-empty. We say that $P_1$ and $P_2$ are \emph{nested} otherwise.
The following holds:

\begin{lemma} \label{spherepartition}
For any clopen partition $P$ of $End(\widetilde{M_g})$ there is a sphere $s_P$ embedded in $\widetilde{M_g}$ inducing that partition. Two embedded spheres are homotopic if and only if they induce the same partition on $End(\widetilde{M_g})$.
\end{lemma}
\begin{lemma} \label{sphereintersection}
Two non-homotopic embedded minimally intersecting spheres $s_1$, $s_2$ in $\widetilde{M_g}$ intersect at most once, they intersect if and only if the partitions induced by $s_1$ and $s_2$ on the space of ends of $\widetilde{M_g}$ are non-nested.
\end{lemma}
We refere to Sections 3 and 4 of \citep{GadPan} or to Section 2.1.1 of \citep{Iez} for proofs of Lemma \ref{spherepartition} and Lemma \ref{sphereintersection}.

\section{Standard form for sphere systems, piece decomposition and dual square complexes} \label{Standard form}
In this section we introduce the definition of a \emph{standard form} for sphere systems, and we describe some properties of this standard form.
Standard form is a refinement of Hatcher's normal form. Loosely speaking, two embedded maximal sphere systems in $M_g$ are in standard form if they intersect minimally and, in addition, their complementary components are as simple as possible.
\par To clarify terminology, if $\Sigma_1$ and $\Sigma_2$ are two sphere systems in $M_g$ and $\widetilde{\Sigma_1}$, $\widetilde{\Sigma_2}$ are their full lifts to the universal cover $\widetilde{M_g}$, we say that $\Sigma_1$ and  $\Sigma_2$ are in \emph{minimal form} if each sphere in $\widetilde{\Sigma_1}$
intersects each sphere in $\widetilde{\Sigma_2}$ minimally.
\par A priori, our definition of minimal form seems stronger than
the most intuitive definition (requiring the number of components of $\Sigma_1 \cap \Sigma_2$ to be minimal over the homotopy class of $\Sigma_1$ and $\Sigma_2$); indeed, both Hatcher's work (\citep{Hat1}), and  Theorem \ref{consequence1} below imply that the two definitions of minimality are equivalent; furthermore, if one of the two systems is maximal, then minimal form is equivalent to Hatcher's normal form (Lemma 7.2 in \citep{HenOP}). However, our arguments will not use the equivalence of these definitions, and will be independent on Hatcher's work.

\begin{definition} \label{def standard form}
Let $\Sigma_1$ and  $\Sigma_2$ be two embedded maximal sphere systems in $M_g$. We say that $\Sigma_1$ and  $\Sigma_2$ are in \emph{standard form} if they are in minimal form with respect to each other and moreover all the complementary components of $\Sigma_1 \cup \Sigma_2$ in $M_g$ are handlebodies. We define in the same way standard form for two maximal sphere systems $\widetilde{\Sigma_1}$ and $\widetilde{\Sigma_2}$ in the universal cover $\widetilde{M_g}$.
\end{definition}


\par The existence of a standard form for any two maximal sphere systems can be deduced from Proposition 1.1 of \citep{Hat1}. We give below an alternative proof (Theorem \ref{consequence1}). We also prove a sort of uniqueness for standard form (Theorem \ref{consequence2}).

\subsection{Piece decomposition}

Given the manifold $M_g$ and two embedded maximal sphere systems $\Sigma_1$ and $\Sigma_2$ in standard form, we coluor $\Sigma_1$ with black and $\Sigma_2$ with red.  We will  call the components of $M_g \setminus (\Sigma_1 \cup \Sigma_2)$ the \emph{3-pieces} of $(M_g, \Sigma_1, \Sigma_2)$, we call the components of $\Sigma_1\setminus \Sigma_2$ the \emph{2-pieces of $\Sigma_1$}, or the \emph{black 2-pieces}, and the components of $\Sigma_2\setminus \Sigma_1$ the \emph{2-pieces of $\Sigma_2$}, or the \emph{red 2-pieces}. Finally, we call the components of $\Sigma_1 \cap \Sigma_2$ the \emph{1-pieces} of $(M_g, \Sigma_1, \Sigma_2)$.
The manifold $M_g$ is the union of $1$-pieces, $2$-pieces and $3$-pieces. We call this collection of pieces a \emph{piece decomposition} for the triple $(M_g, \Sigma_1, \Sigma_2)$.
In the same way we can define a piece decomposition for the triple $(\widetilde{M_g}, \widetilde{\Sigma_1}, \widetilde{\Sigma_2})$. Note that, since the complementary components of maximal sphere systems are simply connected,  pieces of $(M_g, \Sigma_1, \Sigma_2)$ lift homeomorphically to pieces of $(\widetilde{M_g}, \widetilde{\Sigma_1}, \widetilde{\Sigma_2})$; hence, to study the properties of a piece decomposition, we can analyse pieces of $(M_g, \Sigma_1, \Sigma_2)$ or of $(\widetilde{M_g}, \widetilde{\Sigma_1}, \widetilde{\Sigma_2})$, according to what is most convienent for our aim.
In this section we will describe some features of a piece decomposition for a triple $(M_g, \Sigma_1, \Sigma_2)$.
\vskip 0.2cm
Recall that, by maximality, all the components of $M_g \setminus \Sigma_1$ and $M_g \setminus \Sigma_2$ are 3-holed 3-spheres, and that we are assuming hypothesis ($*$) in the introduction.


\vskip 0.2cm
\par We start analysing 1-pieces; since $\Sigma_1$ and $\Sigma_2$ intersect transversely, all 1-pieces of $(M_g, \Sigma_1, \Sigma_2)$ are circles.
\vskip 0.2cm

Now, let $p$ be a 2-piece of $\widetilde{\Sigma_2}$. First, $p$ is a planar surface (for $p$ is a subsurface of a sphere $\sigma \in \widetilde{\Sigma_2}$, and moreover, by maximality and hypothesis ($*$), $\sigma \cap \widetilde{\Sigma_1}$ is non-empty). Further, $p$ is contained in a component of $\widetilde{M_g} \setminus \widetilde{\Sigma_1}$, and, by minimal form and Lemma \ref{sphereintersection},  $p$ cannot intersect the same sphere of $\widetilde{\Sigma_1}$ in more then one circle, which implies that $p$ has at most three boundary components. Summarising, a red 2-piece $p$  embedded in a component $C$ of $\widetilde{M_g} \setminus \widetilde{\Sigma_1}$ is either a disc, or an annulus, or a pair of pants. If  $p$ is a disc, then, by minimal form, $\partial p$ lies on a boundary component of $C$, and $p$ separates the other two components. If $p$ is an annulus or a pair of pants, then different components of $\partial p$ lie on different components of $\partial C$. The same requirements hold true for black 2-pieces in $\widetilde{M_g}$ (by symmetry), and for 2-pieces of $(M_g, \Sigma_1, \Sigma_2)$.



\par Note that these are the same conditions required by Hatcher's normal form.

\vskip 0.2cm
\par As for $3$-pieces, they are all handlebodies by definition of standard form.
Further, the boundary of a 3-piece is the union of black 2-pieces, red 2-pieces, and 1-pieces; where each 1-piece is adjacent to a black 2-piece and a red 2-pieces. Given a 3-piece $P$, we use the term \emph{boundary pattern} for $P$ to refer to the union of 2-pieces and 1-pieces composing $\partial P$. Note that, since the complementary components of maximal sphere systems are 3-holed 3-spheres, the boundary pattern of a 3-piece contains at most three black 2-pieces and three red 2-pieces.
Indeed, standard form imposes further conditions on boundary patterns and it turns out that there are only nine possibilities, which are drawn, from different perspectives,  on the left hand sides of Figure \ref{possiblelinks} and Figure \ref{TABLE 3}.  In particular, the genus of a 3-piece is at most four.
We omit a proof of this here, since it would not be essential to the arguments, and refere to Appendix B of \citep{Iez}.

\vskip 0.2cm
\par In the remainder we will often need to work with the closure of 2-pieces and 3-pieces. Therefore, with a little abuse of terminology, sometimes, we will use the terms \virgolap 2-piece\virgolch and \virgolap 3-piece\virgolch also when refering to the closure of the pieces defined above.

\subsection{Dual square complexes}
\label{Dual square complex}

Given a piece decomposition for the triple $({M_g}, \Sigma_1, \Sigma_2)$, we can construct a dual square complex. We will show later that such a square complex satisfies some very special properties.

\subsubsection{Digression on square complexes} \label{Digression on square complexes}
We make a short digression and recall  some basic definitions and facts concerning square complexes.
Recall that a square complex is a cube complex of dimension 2, i.e, 1-cells are unit intervals and 2-cells are unit euclidian squares. Each 2-cell is attached along a loop of four 1-cells.
We can endow a square complex with a path metric by identifying each 1-cell to the unit interval and each 2-cell to the euclidian square $[0,1] \times [0,1]$.
In the remainder we will use the word \virgolap vertex\virgolch to refer to a 0-cell, the word \virgolap edge\virgolch to refer to a 1-cell and the word \virgolap square\virgolch to refer to a 2-cell.
A square complex is said to be V-H (Vertical-Horizontal) if each edge can be labeled as vertical or horizontal, and on the attaching loop of each square vertical and horizontal 1-cells alternate.
\par An important concept is the concept of \virgolap hyperplane\virgolch, which we shortly define below.
First we define a \emph{midsquare}  as a unit interval contained in a square, parallel to one of the edges, containing the baricenter of the square.
Consider now the equivalence relation $\sim$ on the edges of a square complex $\Delta$, generated by $e \sim e'$  if $e$ and $e'$ are opposite edges of the same square in $\Delta$.
Given an equivalence class of edges [e] in $\Delta$ we define the \emph{hyperplane dual to [e]} as the set of midsquares in $\Delta$ intersecting edges in [e]. Note that hyperplanes of a square complex are connected graphs, and that
 if $\Delta$ is a CAT(0) square complex, then two hyperplanes in $\Delta$ intersect at most once.
As a last note, recall that, by a generalisation of Cartan-Hadamard Theorem (\citep{BriHae} p. 193), a simply connected locally CAT(0) metric space is CAT(0), where the term locally CAT(0) means that each point has a CAT(0) neighbourhood.
We refer to \cite{Sag} for some  background on cube complexes and to \citep{BriHae} for a more detailed discussion on CAT(0) metric spaces.
\vskip 0.3cm

Now, given a triple $(M_g, \Sigma_1, \Sigma_2)$, with $\Sigma_1, \Sigma_2$ in standard form, we can naturally construct a square complex:
vertices correspond to 3-pieces, edges  correspond to 2-pieces and  squares correspond to 1-pieces. For $n=1,2$, an $n$-cell is attached to an $(n-1)$-cell if the piece corresponding to the former lies on the boundary of the piece corresponding to the latter. We denote this complex as $\Delta (M_g, \Sigma_1, \Sigma_2)$ and call it the square complex \emph{dual} to $\Sigma_1$ and $\Sigma_2$ in $M_g$.
We colour with black the edges corresponding to 2-pieces of $\Sigma_1$ and with red the edges corresponding to 2-pieces of $\Sigma_2$.
In the same way we can construct a square complex
$\Delta(\widetilde{M_g},\widetilde{\Sigma_1}, \widetilde{\Sigma_2})$, dual to the triple
$(\widetilde{M_g},\widetilde{\Sigma_1}, \widetilde{\Sigma_2})$.
\par Note that the action of the free group $F_g$ on the manifold
$\widetilde{M_g}$ induces an action of the same group on the complex $\Delta(\widetilde{M_g}, \widetilde{\Sigma_1}, \widetilde{\Sigma_2})$ and the quotient of this square complex under the group action is the square complex  $\Delta(M_g,\Sigma_1, \Sigma_2)$.
\par When no ambiguity can occur we will just denote the complex $\Delta(M_g, \Sigma_1, \Sigma_2)$ as $\Delta$, and
the complex $\Delta(\widetilde{M_g}, \widetilde{\Sigma_1},
\widetilde{\Sigma_2})$ as $\widetilde{\Delta}$.
\par Next we state some properties the complex $\widetilde{\Delta}$ must satisfy (Lemma \ref{prop 0}-Lemma \ref{prop 4}). Proving these Lemmas using the definition of standard form is not difficult. However we defer the proofs to the next section, where we introduce a more abstract way of constructing $\widetilde{\Delta}$.  We refer to Section 2.2 of \citep{Iez} for a direct proof of  Lemma \ref{prop 0}, Lemma \ref{prop 3} and lemma \ref{prop 2}.

\begin{lemma}\label{prop 0}
$\widetilde{\Delta}$ is a connected, locally finite, V-H square complex.
\end{lemma}

We will treat the black edges as horizontal and the red ones as vertical. We will call vertical (resp. horizontal) lines, the segments parallel to vertical (resp. horizontal) edges.

\begin{lemma} \label{prop 1}
The complex $\widetilde{\Delta}$ is endowed with two surjective projections, $p_1$, $p_2$, onto two infinite trivalent trees, $T_1$ and $T_2$. The projection $p_1$ (resp. $p_2$) corresponds to collapsing the vertical (resp. horizontal) lines to points.
\end{lemma}

Indeed, the trees $T_1$ and $T_2$ correspond to the dual trees to $\Sigma_1$ and $\Sigma_2$.


\begin{lemma} \label{prop 3}
Hyperplanes in $\widetilde{\Delta}$ are finite trees.
\end{lemma}
Indeed, hyperplanes dual to red (resp. black) edges correspond to spheres in
$\widetilde{\Sigma_1}$ (resp.
$\widetilde{\Sigma_2}$). Edges and vertices of a hyperplane correspond respectively to $1$-pieces and $2$-pieces contained in the corresponding sphere.


\vskip 0.2cm
As for a vertex link in  $\widetilde{\Delta}$, this is entirely determined by the boundary pattern of the 3-piece the vertex corresponds to.
Such link is necessarily a subgraphs of the bipartite graph $K_{(3,3)}$, since a boundary pattern contains at most three red 2-pieces and three black 2-pieces, where two 2-pieces of the same colour cannot be adjacent. Restrictions imposed by standard form make it possible to list all possible boundary patterns and corresponding vertex links (which are described in Figure \ref{possiblelinks}). In fact, the following holds:
\begin{lemma} \label{prop 2}
All the possible vertex links for $\widetilde{\Delta}$ are the nine graphs  listed in Figure \ref{possiblelinks}.
\end{lemma}



\begin{figure}
\begin{center}
\includegraphics[scale=0.4]{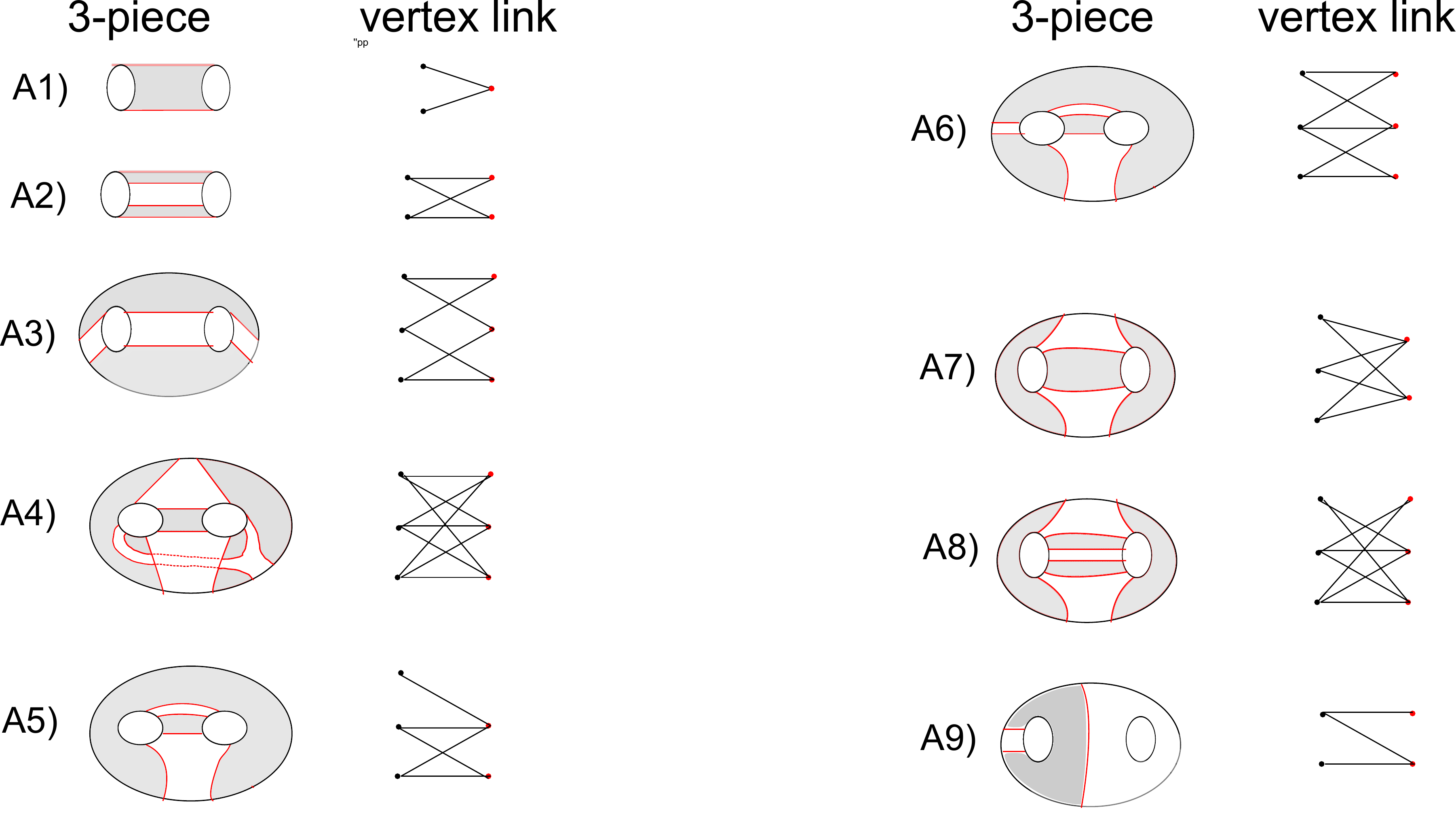}
\caption{All the possibilities for 3-pieces in $M_g$ with the corresponding vertex links. The 3-piece we consider is the part outlined with grey. Pictures are drawn in one dimension less, i.e. we draw a section of each piece, for example a circle represents a sphere, two parallel lines represent an annulus etc. In each picture the three black circles represent three spheres of $\Sigma_1$ bounding a component $C$ of $M_g \setminus \Sigma_1$. The red lines represent the 2-pieces of $\Sigma_2$ in a complementary component  of $\Sigma_1$. }\label{possiblelinks}
\end{center}
\end{figure}


Another important property is the following:
\begin{lemma} \label{prop 4}
The complex $\Delta(\widetilde {M_g}, \widetilde{\Sigma_1}, \widetilde{\Sigma_2})$ is simply connected.
\end{lemma}


We omit a proof of Lemma \ref{prop 4} here, since it will be an immediate consequence of Proposition \ref{C simplyconnected} and Proposition \ref{equalityconstructions}.
Note that simply connectedness of the complex $\widetilde{\Delta}$ implies that $\widetilde{\Delta}$ is the universal cover of the complex $\Delta$ and, as a consequence, the latter complex has the free group $F_g$ as fundamental group.

\section{The core of two trees}
\label{The core of two trees}
In this section we describe an abstract construction. Starting with two trees endowed with geometric actions by the free group $F_g$, we construct a square complex $C$ and call it the \emph{core} of the product of the two trees.
We show later (Proposition \ref{equalityconstructions}) that the method described below provides an alternative way of constructing the square complex dual to two maximal sphere systems in standard form in $\widetilde{M_g}$.
\vskip 0.1cm
Let $T$ and $T'$ be two three-valent trees both endowed with a free, properly discontinuous and cocompact action by the free group $F_g$, call these actions $\rho$ and $\rho'$.

\par Note that the action $\rho$ (resp. $\rho'$) induces a canonical identifications of the group $F_g$ to the Gromov boundary of the tree $T$ (resp. $T'$). Hence the boundaries of  $T$ and $T'$ can also be identified. In view of this identification, in the remainder,  we will often use the term \emph{boundary space}  to refer to the Gromov boundary of $T$ and the Gromov boundary of $T'$.


\vskip 0.1cm
Each edge in $T$ or $T'$ induces a partition on the boundary space.
If $e$ is an edge in $T$, we will denote by $P_e$ the partition induced by $e$ on the boundary space and by $e^+$ and $e^-$ the two sets composing this partition.
If $e$ is an edge in $T$ and $e'$ is an edge in $T'$, we say that the induced partitions $P_e$ and $P_{e'}$ are \emph{non-nested} if no set of one partition is entirely contained in a set of the other partition, namely all  sets $e^+ \cap {e'}^+$, $e^+ \cap {e'}^-$, $e^- \cap {e'}^+$, $e^- \cap {e'}^-$ are non-empty. We say that $P_e$ and $P_{e'}$ are \emph{nested} otherwise.

\par In the remainder we will assume the following hypothesis:
\vskip 0.08cm
\par ($**$) There do not exist an edge $e$ in $T$ and an edge $e'$ in $T'$ inducing the same partition on the boundary space. \label{hypothesis ($**$)}
\vskip 0.08cm
It will be clear in the next section that hypothesis ($**$) above corresponds to hypothesis ($*$) in the introduction.
The construction described below can be generalised to the case where hypothesis ($**$) is not fulfilled. We refer to Section 2.6.3 of \citep{Iez} for a discussion of this more general case.
\vskip 0.1cm
\par Consider now the product $T \times T'$, this is a CAT(0) square complex, where each vertex link is the bipartite graph $K_{3,3}$.
This space can be naturally endowed with a diagonal action $\gamma$.
Namely, given a vertex $(v_1,v_2)$ in $T\times T'$ and an element $g$ of $F_g$, we set $\gamma_g (v_1,v_2)$ to be the vertex $(\rho_g(v_1),\rho'_g(v_2))$. Since $\rho$ and $\rho'$ are free and properly discontinuous, so is $\gamma$.
We go on to define the main object of this section.
\begin{definition}
The \emph{core} of $T$ and $T'$ is the subcomplex of $T \times T'$ consisting of all the squares $e \times e'$ where $e$ is an edge in $T$, $e'$ is an edge in $T'$, and the two partitions induced by $e$ and $e'$ on the boundaries of $T$ and $T'$ are non-nested. We will denote this complex by $C(T,T')$. Where no ambiguity can occur, we will write $C$ instead of $C(T,T')$.
\end{definition}



The complex $C$ is invariant under the diagonal action of the group $F_g$.
In fact, for each $g$ in $F_g$ the maps $\rho_g$ and $\rho'_g$ induce the same homeomorphism on $\partial T= \partial T'$, therefore, the partitions induced by the edges $e$ and $e'$ are nested if and only if the ones induced by the edges $\rho_g(e)$ and $\rho'_g(e')$ are.
Denote by $\Delta(T, T')$  the quotient of $C(T, T')$ by
the diagonal action of the group $F_g$.

\begin{remark}
Note that, in order to define the complex $C(T,T')$, we do not need the group actions but only an identification of $\partial T$ and $\partial T'$. Furthermore, the construction can be applied to any pair of simplicial trees. In this sense, the construction described above, gives a combinatorial way of defining the Guirardel core of two simplicial trees.
\end{remark}

\subsection{Properties of the core}
\label{properties of the core}
This section is aimed at showing that the complex $C(T,T')$ satisfies Lemmas \ref{prop 0}-\ref{prop 4}.

First note that, as a subcomplex of $T \times T'$, the complex $C(T, T')$ is a V-H square complex, and is naturally endowed with two projections: $\pi_T :C \rightarrow T$ and $\pi_{T\rq{}} :C \rightarrow T\rq{}$.
If $e$ is an edge in $T$ (resp. $e'$ is an edge in $T'$), we use the term \emph{preimage} of  $e$ (resp. of $e'$), to denote
preimage of the edge $e$ under the map $\pi_T :C \rightarrow T$ (resp. of the edge
$e\rq{}$ under the map $\pi_{T\rq{}} :C \rightarrow T\rq{}$). In the same way we define the preimages of vertices. The next goal is to prove the following:

\begin{theorem} \label{core connected quotient finite}
The core $C(T,T')$ is a connected square complex and the quotient $\Delta(T,T')$ is finite.
\end{theorem}
Note that Theorem \ref{core connected quotient finite} can be deduced from \citep{Gui}, using the fact the the $C(T,T')$ coincides with the Guirardel core of $T$ and $T'$. For the sake of completeness, we give below an independent proof, using mostly combinatorial methods. The proof will consist of several steps.
\vskip 0.2cm
To establish terminology, given an edge $e \in T$, we denote by $T_e '$ the subset of $T'$ consisting of edges $e'$ in $T'$ such that the partitions induced by $e$ and $e'$ are not nested. Note that $T_e '$ coincides to the hyperplane dual to $e$ and that the preimage of an edge $e$ is exactly the interval bundle over $T_e'$.
We will prove the following:

\begin{lemma} \label {edge fiber connected}
For each edge $e \in T$, $T_e'$ is a finite  subtree of $T'$.
\end{lemma}

\begin{proof}

By definition, given $e \in T$, the edge $e'$ belongs to $T_e'$  if and only if all the sets ${e'}^+ \cap e^+$, ${e'}^- \cap e^+$, ${e'}^+ \cap e^-$ and ${e'}^- \cap e^-$ are non-empty.
We first prove that, $T_e'$ is connected, by showing that if two edges $a$ and $b$ in $T'$ belong to $T_e'$, then the geodesic in $T'$ joining $a$ and $b$ is contained in $T_e'$. To prove this, we may suppose
$a^+ \supset c^+ \supset b^+$, and as a consequence $a^- \subset c^- \subset b^-$.
Now, since the sets $b^+ \cap e^+$ and $b^+ \cap e^-$ are non-empty, then the sets $c^+ \cap e^+$ and $c^+ \cap e^-$ are respectively non-empty; since the sets $a^- \cap e^+$ and $a^- \cap e^-$ are non-empty, then the sets $c^- \cap e^+$ and $c^- \cap e^-$ are also non-empty.
Hence, the edge $c$ belongs to $T_e'$.
\vskip 0.1cm
Next, we prove that $T_e'$ is finite.
By connectedness of $T_e'$, it is sufficient to show that, if $r'= \{e'_i\}$ with $i \in \mathbb{N}$  is a geodesic ray in $T'$, then the subset of $r'$ contained in $T_e'$ is finite, i.e.  there exists $I \in \mathbb{N}$ such that the set ${e'_i} ^+$ (or ${e'_i} ^-$) is contained in one of the two sets $e^+$ or $e^-$ for each $i \geq I$.
To prove that, suppose $e_{i+1} ^+ \subset e_i^+$, and note that the limit of the ray $\{e'_i\}$ is a point $p$ in $\partial T'$; we may assume $p$ belongs to $e^+$. Now suppose that both sets $e_i ^+ \cap e^+$ and $e_i ^+ \cap e^-$ are non-empty for each  $i$; then, since these are compact subsets of the compact space $\partial T'$, both sets $(\cap_{i \in \mathbb{N}} e_i ^+) \cap e^+$
and $(\cap_{i \in \mathbb{N}} e_i ^+) \cap e^-$ are non-empty, which leads to a contradiction, since $(\cap_{i \in \mathbb{N}} e_i ^+)= p \in e^+$.
\end{proof}

\begin{remark} \label{fin hyperplanes}
As an immediate consequence of Lemma \ref {edge fiber connected}, all the preimages of edges in $T$ (and by symmetry all the preimages of edges $e' \in T'$) are finite and connected, since they are interval bundles over finite trees, and
hyperplanes in $C(T, T')$ are finite trees.
\end{remark}

As another consequence of Lemma \ref{edge fiber connected} we can prove the following:

\begin{proposition} \label{core finite}
The quotient space $\Delta(T, T')$ is a finite V-H square complex, and hyperplanes in $\Delta(T, T')$ are trees.
\end{proposition}
\begin{proof}
Denote as usual by $\rho$ and $\rho'$ the actions of $F_g$ on $T$ and $T'$ respectively, and by $\gamma$ the diagonal action of $F_g$ on $T \times T'$.
Note that, by invariance of the core, if $e$ is any edge in $T$, $g$ is any element of $F_g$, and $F_e$ is the preimage of the edge $e$ in $C(T,T')$, then the preimage $F_{\rho_g(e)}$ of the edge $\rho_g(e)$ is exactly $\gamma_g(F_e)$. Note also that, $T/F_g$ is a finite graph, by cocompactness of the action $\rho$.
Hence, finiteness of edge preimages immediately implies that the complex $\Delta(T, T')$ is finite.
\par Now, $\Delta(T, T')$ is a V-H square complex, since $C(T,T')$ is, and the diagonal action $\gamma$ maps vertical (resp. horizontal) edges to vertical (resp. horizontal) edges.
\par To prove that hyperplanes in $\Delta(T, T')$ are trees, we observe that, since the actions $\rho$ and $\rho'$ are free, two squares belonging to the preimage of the same edge cannot be identified under the quotient map, therefore, the restriction of the quotient map to a hyperplene is a graph isomorphism.
\end{proof}

To prove connectedness of the core $C(T,T')$ we still need some preliminary lemmas. While proving the next lemma, we use hypothesis ($**$), i. e. we suppose that there do not exist an edge in $T$ and an edge in $T'$ inducing the same partition on $\partial T$.

\begin{lemma} \label{connectedness claim 1}
The projections $\pi_T :C \rightarrow T$ and $\pi_{T'} :C \rightarrow T'$ are both surjective.
\end{lemma}
\begin{proof}
We prove that $\pi_T$ is surjective, i. e. for each edge $e$ in $T$ there exists an edge $e'$ in $T'$ such that $e \times e'$ is in $C$ or equivalently such that the partition induced by $e$ and the one induced by $e'$ are non-nested. The same argument can be used to prove that the projection $\pi_{T'}$ is surjective. Let $e$ be any edge in $T$.
As usual we denote by $P_e = e^+ \cup e^-$ the partition induced by the edge $e$.
\vskip 0.07cm
\par First we claim that there are edges $a$ and $b$ in $T'$ such that $a^+ \subset e^+ \subset b^+$.
To prove the claim note first that, since by Remark \ref{fin hyperplanes} the preimage $F_e$ is finite, there exists at least one edge $a$ in $T'$ such that the partitions $P_e$ and $P_a$ are nested.
We may suppose without losing generality $a^+ \subset e^+$. Now, pick a point $p$ in $e^-$ and let $r = a,e'_1, e'_2...$ be the geodesic ray in $T'$ joining the edge $a$ to the point $p$. Since we have $e^- \subset a^-$, the point $p$ belongs to $a^-$. Consequently we have the containment ${e'}_i ^+ \subset {e'} _{i+1} ^+ $ for each ${e_i}' $ in the geodesic ray $r$. The set $\bigcup _i {e'} _i ^+$ coincides with $\partial T' \setminus p$, and therefore it contains $e^+$. Since $e^+$ is compact, there exists a natural number $I$ such that ${e'} _i ^+$ contains $e^+$ for each $i$ greater than or equal to $I$. We choose $I$ to be minimal among the numbers having this property. The edge $e'_I$ is what we were looking for. Denote $e'_I$ by $b$. This proves the claim.
\vskip 0.07cm
\par We will now use this claim to find an edge $e'$ in $T'$ such that the partitions induced by $e$ and $e'$ are non-nested. If we use the same notation as above this edge will be one of the two edges adjacent to $b$.
\par \noindent Let $a$ and $b$ be the edges of $T'$ defined above and let $r$ be the geodesic defined above. Denote by $c$ the edge immediately preceding $b$ on the ray $r$. We know that $c^+ \subset b^+$, that $b^+ \supset e^+$, and that $c^+$ does not contain $e^+$ (this follows from the fact that $b$ is the first edge in the ray $r$ such that $b^+$ contains $e^+$). Recall that, by hypothesis ($**$), no containment can be an equality. There are then two possibilities.
\par -Case 1: $c^+  \nsubseteq e^+$. In this case the partition induced by $e$ and the one induced by $c$ are non-nested.
\par -Case 2: $c^+  \subset e^+$. In this case call $d$ the edge in $T'$ adjacent to both $b$ and $c$ and denote by $v$ the vertex in $T'$ where the edges $c$, $b$, $d$ intersect. We claim that the partitions induced by $d$ and $e$ are not nested. To prove the claim, first observe that the vertex $v$ induces a partition $\partial T' = D_1 \cup D_2 \cup D_3$ where $D_1$ is equal to $c^+$, $D_2$ is equal to $b^-$ and $D_3$ is equal to $\partial  T' \setminus(D_1 \cup D_2)$; hence the partition induced by $d$ is $\partial T' = (D_1 \cup D_2) \cup D_3$.
Now, since $e^+$ strictly contains $c^+ = D_1$ and is strictly contained in $D_1 \cup D_3 = b^+$, both sets $e^+ \cap D_3$ and $e^- \cap D_3$ are non-empty.
Moreover, since $D_1$ is contained in $e^+$ the set $(D_1 \cup D_2) \cap e^+$ is non-empty; and since $e^+$ is contained in $b^+ = D_1 \cup D_3$ then $D_2$ is contained in $e^-$ and therefore the set $(D_1 \cup D_2) \cap e^-$ is non-empty.
Consequently the partitions induced by $e$ and $d$ are non-nested, which concludes the proof.
\end{proof}

We have already shown that the preimage of each edge through the projections $\pi_T$ and $\pi_{T'}$ is connected. The following lemma will allow us to conclude the proof of the connectedness of the core.

\begin{lemma} \label{connectedness claim 2}
The preimage of each vertex is a finite tree, in particular it is connected.
\end{lemma}
\begin{proof}
Let $v \in T$ be a vertex, denote by $F_v$ its preimage through $\pi_T :C \rightarrow T$. We show that $F_v$ is a finite tree. By symmetry, the same arguments show that for any vertex $v' \in T'$, its preimage through $\pi_{T'}$ is a finite tree.
\par Note first that, if we denote by $e_1$, $e_2$ and $e_3$ the three edges incident to the vertex $v$, then the preimage $F_v$
is the union $T_{e_1}' \cup T_{e_2}' \cup T_{e_3}'$ (recall that $T_{e_i}'$ is the subtree of $T'$ consisting of all the edges $e'$ so that the partitions induced by $e'$ and $e_i$ are non-nested).  Hence $F_v$ is the union of three finite trees; we go on to prove that $F_v$ is connected.

\par To reach this goal, we first state a necessary and sufficient condition for an edge in $T'$ to belong to $F_v$. Then we show that the set of edges in $T\rq{}$ satisfying this condition is connected.
\par  To state the condition, we first observe that $v$ induces a partition of $\partial T$ given by $\partial T =D_1 \cup D_2 \cup D_3$. If $e_1$, $e_2$ and $e_3$ are as above then $e_1$ induces the partition $\partial T =D_1 \cup (D_2 \cup D_3)$, $e_2$ induces the partition $\partial T =D_2 \cup (D_1 \cup D_3)$ and $e_3$ induces the partition $\partial T =(D_1 \cup D_2) \cup D_3$.
We claim that an edge $e'$ in $T'$ belongs to $F_v$ if and only if neither of the sets ${e'} ^+$ and ${e'} ^-$ is entirely contained in any of the $D_i$s. Note that hypothesis ($**$) above implies that it is not possible to have ${e'}^+ =D_i$ or ${e'}^- =D_i$ for any $i$. We go on to prove the claim.
\par A direction  is straightforward, in fact if ${e'} ^+$ (or ${e'} ^-$) is contained in one of the $D_i$s then the partition induced by $e'$ and the one induced by $e_j$ would be nested for each $j=1,2,3$.
Let us prove now that, if for each $i=1,2,3$ we have ${e'}^+ \nsubseteq D_i$ and ${e'}^- \nsubseteq D_i$ then there exists an $i$ such that the partition induced by $e'$ and the one induced by $e_i$ are not nested, i. e. there exists an $i$ such that all the sets ${e'}^+ \cap D_i$, ${e'}^- \cap D_i$, ${e'}^+ \cap {D_i}^C$,
${e'}^- \cap {D_i}^C$ are non-empty.
To prove this, note first that there exists an $i$ such that both sets ${e'}^+ \cap D_i$ and ${e'}^- \cap D_i$ are non-empty; in fact, if this were not true, there would exist an $i$ such that either $e\rq{}^+ =D_i$ or $e\rq{}^- =D_i$,  contradicting hypothesis ($**$).
Now, since ${e'}^+$ is not entirely contained in $D_i$, then the set ${e'}^+ \cap {D_i}^C$ is also non-empty, and since ${e'}^-$ is not entirely contained in $D_i$, then the set ${e'}^- \cap {D_i}^C$ is non-empty, which proves the claim
\vskip 0.07cm
To prove that the set of edges in $T'$ satisfying this condition is connected,
we use a similar argument as in the proof of Lemma \ref{fin hyperplanes}: we show that if two edges $a$ and $b$ in $T'$ belong to $F_v$ then the geodesic segment in $T'$ joining $a$ and $b$ is contained in $F_v$. Suppose $a^+ \supset b^+$, then for any edge $c$ in
the geodesic segment joining $a$ and $b$, we have $a^+ \supset c^+ \supset b^+$ and
$a^- \subset c^- \subset b^-$. Since  $b^+$ is not contained in any of the $D_i$s, neither is $c^+$, and since $a^-$ is not contained in any of the $D_i$s, neither is $c^-$, i. e.  $c$ belongs to $F_v$, which concludes the proof.
\end{proof}

Connectedness of the complex $C(T, T')$ is now an obvious consequence of Lemma \ref{connectedness claim 1}, Lemma \ref{connectedness claim 2}, and Lemma \ref{edge fiber connected}.

Another consequence of Lemma \ref{edge fiber connected} and Lemma \ref{connectedness claim 2} is the following:

\begin{proposition} \label{C simplyconnected}
The complex $C$ is simply connected
\end{proposition}

\begin{proof}
Any loop $l$ in $C$ would project, by compactness, to a finite subtree of $T$, therefore, in order to prove that $C$ is simply connected, it is sufficient to show that, for any finite subtree $S$ of $T$, the preimage $\pi_T ^{-1} (S)$ (which we denote as $F_S$), is simply connected.
\par Now note that, if $S$ is a finite subtree of $T$, then $F_S$ is the union over all the edges $e$  and all the vertices $v$ in $S$ of the preimages $F_e$ and $F_v$, which are all simply connected by Lemma \ref{edge fiber connected} and Lemma \ref{connectedness claim 2}. Moreover, distinct preimages intersect, if at all, in a finite tree. Now, simple connectedness of $F_S$ can be proven by using induction on the number of edges in $S$, and Van Kampen's theorem.
\end{proof}

As a consequence of Proposition \ref{C simplyconnected}, the fundamental group of the quotient $\Delta(T, T')$ is the free group $F_g$.
\vskip 0.1cm

Next we try to understand how the vertex links in $C(T,T')$ look like.  \label{pagetable5}
First note that they are all subgraphs of the bipartite graph $K_{3,3}$, since the complex $C(T,T')$ is contained in the product $T \times T'$. We show next that all possible vertex links in $C(T, T')$ are the nine graphs described in Figure \ref{TABLE 5}, which coincide with the nine graphs listed in Figure \ref{possiblelinks}.

\par Consider a vertex $(v,v')$ in $T \times T'$, and let us try to understand what its link in $C(T, T')$ is (this link is empty in the case where $(v,v')$ is not in $C(T, T')$).

Denote by $e_1$, $e_2$ and $e_3$ (resp. $e_1'$, $e_2'$ and $e_3'$) the three edges incident to $v$ in $T$ (resp. to $v'$ in $T'$). Recall that $v$ induces a partition $\partial T = D_1 \cup D_2 \cup D_3$, where the edge $e_1$ induces the partition $\partial T = D_1 \cup (D_2 \cup D_3)$;
the edge $e_2$ induces the partition $\partial T = D_2 \cup (D_1 \cup D_3)$; the edge $e_3$ induces the partition $\partial T = (D_1 \cup D_2) \cup D_3$; 
the same holds for the vertex $v'$ in $T'$.
Now, $(v,v')$ is incident to nine squares in $T \times T'$, and understanding how the link of $(v,v')$ looks like boils down to understanding which of the squares $e_i \times e_j$ belong to the core, i. e. for which $i$s and $j$s the partitions induced by $e_i$ and $e_j$ are non-nested.
Namely, the link of the vertex $(v,v')$
will consist of two sets of at most three vertices: a black set representing the edges $e_1$, $e_2$ and $e_3$ and a red set representing the edges $e_1'$, $e_2'$ and $e_3'$, where the $i$th black vertex and the $j$th red vertex are adjacent if and only if  the partitions induces by the edges $e_i$ and $e_j'$ are not nested.

\par We can deduce the link of $(v,v')$ by analysing a simple $3 \times 3$ table. This table has a cross in the slot $(i,j)$ if the set $D_i \cap D_j '$ is non-empty, and a circle in the slot $(i,j)$ if the set $D_i \cap D_j '$ is empty. In the caption to Figure \ref{TABLE 5} we explain how to deduce from the position of crosses and circles whether, for $i,j=1,2,3$, the partitions induced by the edges $e_i$ and $e_j'$ are nested.

\par  It is not difficult to analyse systematically all the possible vertex tables.  These are 3 by 3 tables whose entries can be only crosses or circles. Moreover, they have to satisfy some additional condition: first, since $\partial T = D_1 \cup D_2 \cup D_3 = D_1' \cup D_2' \cup D_3'$, there has to be at least a cross in each row and column of the table; second, by hypothesis ($**$),
the union of a row and a column must contain at least two crosses (see Figure \ref{forbidden cases} for an example of these \virgolap forbidden patterns\virgolch). Furthermore permuting the order of rows or columns in the table, or reflecting the table through the diagonal would not change the vertex link.

\begin{figure}
\begin{center}
\includegraphics[scale=0.2]{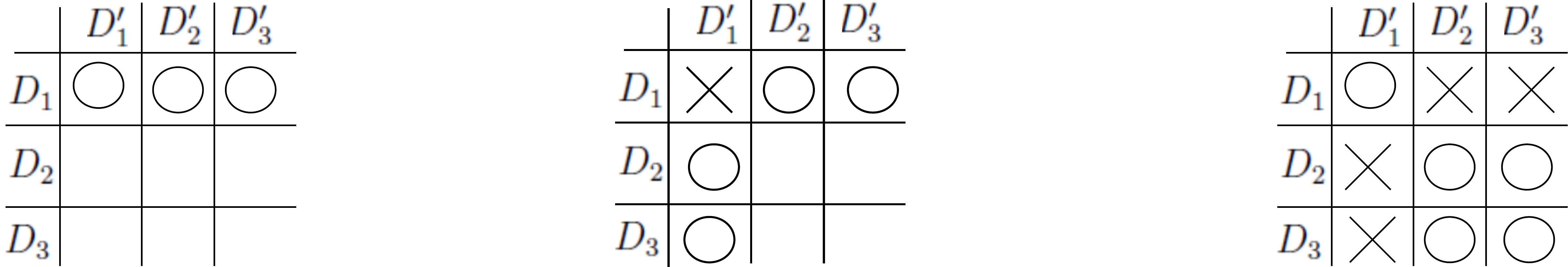}
\caption{The patterns drawn above are \virgolap forbidden\virgolch in vertex tables. In fact, the pattern on
the left hand side of the figure would imply that $\partial T$ is empty;  the pattern in the center would imply $D_1= D_1'$, consequently $e_1$ and $e_1'$ would induce the same partition, contradicting hypothesis ($**$);  the pattern on
the right hand side of the figure would imply $D_1= D_2' \cup D_3'$,  contradicting again hypothesis ($**$)
}\label{forbidden cases}
\end{center}
\end{figure}

\begin{figure}
\begin{center}
\includegraphics[scale=0.25]{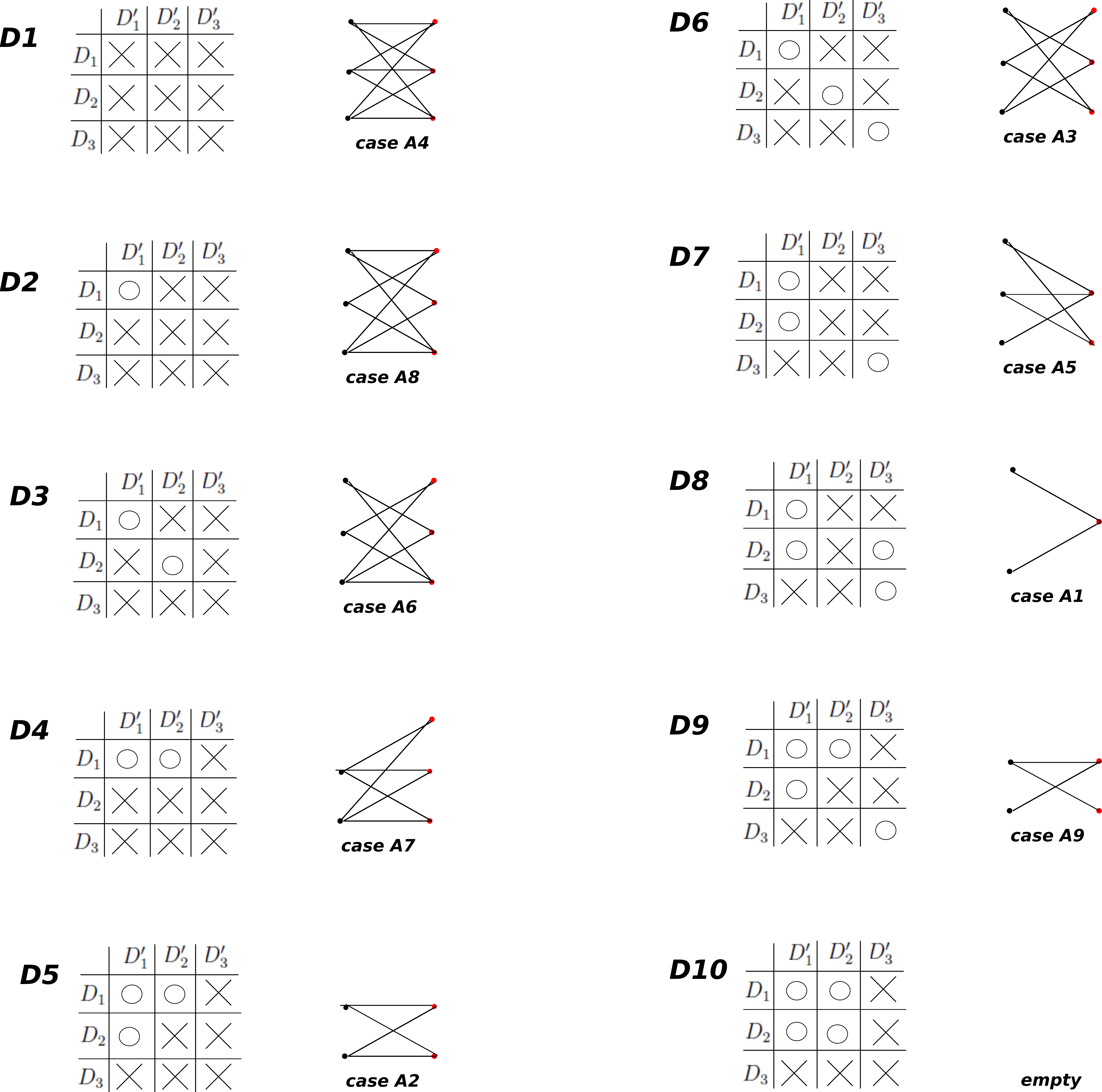}
\caption{This figure describes all the possible vertex tables.
As mentioned, for a vertex $(v,v')$ in $T \times T'$ we draw a $3 \times 3$ table.
The slot $(i,j)$ contains a cross if the set $D_i\cap D_j'$ is non-empty and a circle otherwise. The partitions induced by the edges $e_i$ and $e_j'$ are non-nested if and only if
the table corresponding to $(v,v')$ satisfies the following four properties: the slot $(i,j)$ contains a cross; the row $i$ contains at least another cross; the column $j$ contains at least another cross;  the complement of the row $i$ and the column $j$ contains at least one cross.
At the right hand side of each vertex table we draw the vertex link. We relate the graphs listed here to the ones listed in Figure \ref{possiblelinks} } \label{TABLE 5}
\end{center}
\end{figure}

\par Figure \ref{TABLE 5} gives an exhaustive list of all such $3$ by $3$ tables up to permutation of  rows or columns and reflection around the diagonal, and the vertex link associated to each table.

\vskip 0.2cm
Summarising, given two trivalent trees $T$ and $T'$ endowed with actions by the group $F_g$, the core $C(T,T')$ satisfies the following properties:   \label{properties core}
\par (1) $C(T,T')$ is a connected V-H locally finite square complex
\par (2) hyperplanes in $C(T,T')$ are finite trees
\par (3) all possible vertex links in $C(T,T')$ are the nine graphs listed in Figure \ref{TABLE 5} (which, in particular, implies that $C(T,T')$ is locally CAT(0))
\par (4) $C(T,T')$ is simply connected and endowed with a free properly discontinuous action by the free group $F_g$
\par (5) collapsing the vertical or horizontal lines to points yields projections of $C(T,T')$ to two trivalent trees
\vskip 0.1cm
On the other hand, the complex $\Delta(T, T')$ satisfies properties 2) and 3), and the following:
\par (1') $\Delta$ is a connected V-H finite square complex
\par (4') the fundamental group of $(\Delta(T, T')$ is the free group $F_g$
\par (5') collapsing the vertical or horizontal lines to points yields projections of $C(T,T')$ to two trivalent graphs

\vskip 0.2cm
Note that  properties (1)-(5) above are exactly the properties stated in Lemma \ref{prop 0}-Lemma \ref{prop 4}.
\par Indeed, as a consequence of Proposition \ref{equalityconstructions} and of the construction described in Section \ref{Inverse construction}, the above properties identify the core of two trivalent trees, i. e.
any square complex satisfying properties (1)-(5) can be realised as the core of two trivalent trees.

\subsection{Equality of the two constructions}
In this section we relate the construction of Section \ref{Dual square complex} to the one of Section \ref{The core of two trees}. Namely, we show that the square complex dual to two maximal sphere systems in standard form can always be realised as the core of two trivalent trees. A different proof of this result can be found in \citep{Hor} (Proposition 2.1).
\vskip 0.2cm
\par As usual, let $\Sigma_1$, $\Sigma_2$ be two maximal sphere systems in the manifold $M_g$ in standard form and let $\widetilde{\Sigma_1}$, $\widetilde{\Sigma_2}$ be the lifts to the universal cover $\widetilde{M_g}$. Let $T_1$ and $T_2$ be the dual trees to $\Sigma_1$ and $\Sigma_2$ respectively. Note that the trees $T_1$ and $T_2$ are both trivalent and endowed with a free properly discontinuous action by the group $F_g$ (induced by the action of $F_g$ on $\widetilde{M_g}$), and the product $T_1 \times T_2$ is endowed with the diagonal action.
The complex $\Delta(\widetilde{M_g}, \widetilde{\Sigma_1}, \widetilde{\Sigma_2})$ is also endowed with a free properly discontinuous action of the free group $F_g$, induced by the action of $F_g$ on the manifold $\widetilde{M_g}$.
We prove the following:

\begin{proposition} \label{equalityconstructions}
The complex $\Delta(\widetilde{M_g}, \widetilde{\Sigma_1}, \widetilde{\Sigma_2})$ is isomorphic to the core $C(T_1,T_2)$.
\end{proposition}

\begin{proof}
For the remainder of the proof we denote the square complex $\Delta(\widetilde{M_g}, \widetilde{\Sigma_1}, \widetilde{\Sigma_2})$ by $\widetilde{\Delta}$. We first claim that
$\widetilde{\Delta}$ can be identified to a subcomplex of the product $T_1 \times T_2$

In fact, collapsing the vertical (resp. horizontal) lines to points yields equivariant projections $p_1 :\widetilde{\Delta} \rightarrow T_1$  (resp. $p_2 :\widetilde{\Delta} \rightarrow T_2$): a square in $\widetilde{\Delta}$, representing an intersection between a sphere $\sigma_1$ in $\widetilde{\Sigma_1}$ and a sphere $\sigma_2$ in $\widetilde{\Sigma_2}$, gets projected through $p_1$ to the edge $e_1 \in T_1$ representing $\sigma_1$ and through $p_2$  to the edge $e_2 \in T_2$ representing $\sigma_2$.  Moreover, since, by standard form, a sphere in $\widetilde{\Sigma_1}$ and a sphere in $\widetilde{\Sigma_2}$ can intersect at most once, then, given edges $e_1 \in T_1$ and $e_2 \in T_2$, there is at most one square $s$ in $\widetilde{\Delta}$ so that $p_1(s)$ is $e_1$ and $p_2(s)$ is $e_2$. Hence the claim holds.
\vskip 0.1cm
To conclude the proof, we show that each square in $\widetilde{\Delta}$ is contained in $C(T_1, T_2)$ and vice versa. Let $\sigma_1$ be a sphere in $\widetilde{\Sigma_1}$, $\sigma_2$ be a sphere in $\widetilde{\Sigma_2}$, and let $e_1$ and $e_2$ be the edges of $T_1$ and $T_2$ representing $\sigma_1$ and $\sigma_2$ respectively. Now, the square $e_1 \times e_2$  is in $\widetilde{\Delta}$ if and only if $\sigma_1 \cap \sigma_2$ in non-empty, if and only if (Lemma \ref{sphereintersection}) the partitions induced by $\sigma_1$ and $\sigma_2$ on $End(\widetilde{M_g})$ are non-nested, if and only if the partitions induced by $e_1$ and $e_2$ on the boundary $\partial T_1 = \partial T_2$ are non-nested, if and only if $e_1 \times e_2$ is in $C(T_1, T_2)$.
\end{proof}

As a consequence, the complex $\Delta(\widetilde{M_g},\widetilde{\Sigma_1}, \widetilde{\Sigma_1})$ satisfies properties (1)-(5).
Hence Lemmas \ref{prop 0}-\ref{prop 4} hold.

Note that, if we see $\Delta(\widetilde{M_g},\widetilde{\Sigma_1}, \widetilde{\Sigma_1})$ as a subcomplex of $T_1 \times T_2$, then the $F_g$-action on this square complex induced by the $F_g$-action on the manifold $\widetilde{M_g}$ coincides with the diagonal action of $F_g$ on the product $T_1 \times T_2$. Therefore an immediate consequence of Theorem \ref{equalityconstructions} is the following:

\begin{corollary} \label{cor equality constructions}
The square complex $\Delta(M_g, \Sigma_1,\Sigma_2)$ is isomorphic to the square complex $\Delta(T_1, T_2)$.
\end{corollary}

As a consequence of Corollary \ref{cor equality constructions}, the complex $\Delta(M_g, \Sigma_1,\Sigma_2)$ satisfies properties (2)-(3), (1')-(5') above.
\vskip 0.2cm

In the next section we will prove a sort of converse to Theorem \ref{equalityconstructions}, i.e. the core of two trivalent trees both endowed with actions by the free group $F_g$ can always be realised as the square complex dual to two maximal sphere systems embedded in the manifold $\widetilde{M_g}$.

\section{The inverse construction}
\label{Inverse construction}
In Section \ref{Dual square complex}, given a manifold $M_g$ with two embedded sphere systems in standard form, we have decomposed the manifold as a union of simple pieces, we have described a dual square complex and proved it satisfies properties (2), (3), (1')-(5') in Section \ref{The core of two trees}.
In this section we are going to describe an inverse procedure. Namely, starting with a square complex satisfying the above properties, we will associate to each cell of the complex a \virgolap piece\virgolch, together with some \virgolap gluing rules\virgolch. We will show then that this provides a piece decomposition for a manifold $M_g$ and two embedded sphere systems in standard form.
\vskip  0.3cm
Let $\Delta$ be a square complex satisfying properties (2), (3), (1')-(5') in Section \ref{The core of two trees}. We associate to each cell of ${\Delta}$ a piece in the following way.
\vskip 0.3cm

\vskip 0.1cm
\par Given a square $s$ in ${\Delta}$ we associate to $s$ a circle, $c(s)$. We call these circles \emph{1-pieces}.
\vskip 0.3cm
\par As for edges, we will refer to horizontal edges as \emph{black edges} and to vertical edges as \emph{red edges}.  We associate to each edge $e$ of ${\Delta}$ a planar surface $p(e)$ having as many boundary components as the number of squares adjacent to $e$ (which is at most 3 by property (5')), as shown in Figure \ref{TABLE 2}. We colour these surfaces with black or red, according to the colour of the edge they are associated to, and we call these surfaces \emph{2-pieces}.

\begin{figure}
 \begin{center}
\includegraphics[scale=0.15]{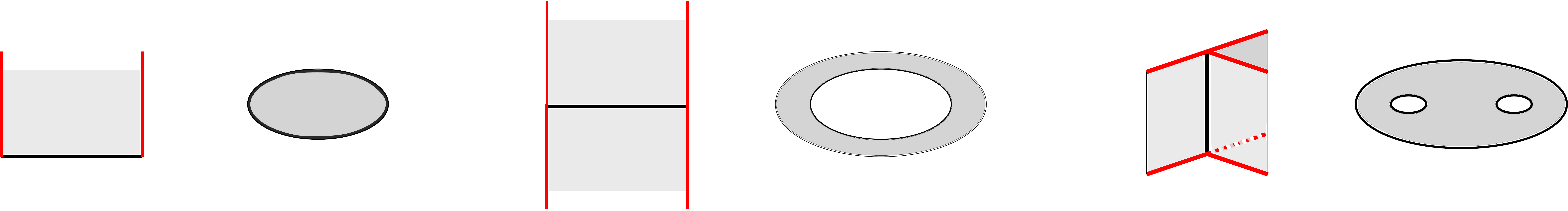}
\caption{How to associate a 2-piece $p(e)$ to an edge $e$. The edge we consider is the black edge in the picture. The associated 2-piece is a disc if the edge bounds one square, an annulus if the edge bounds two squares, and a pair of pants if the edge bounds three squares.} \label{TABLE 2}
 \end{center}

\end{figure}

\vskip 0.3cm
\par We associate to each vertex $v$ a closed surface $S(V)$ and a handlebody $P (v)$ according to the link of $v$ in $\Delta$. The surface $S(v)$ is the union of the 2-pieces and the 1-pieces
associated to the edges and squares incident to $v$;
the vertex link determines how these pieces are glued together, as described in Figure \ref{TABLE 3}. Then $P(v)$ is the handlebody having $S(v)$ as its boundary and (if $S(v)$ has positive genus) such that  $1$-pieces in $S(v)$ are not all trivial in the fundamental group of $P(v)$.
We call the handlebody $P(v)$ a \emph{3-piece} and the surface $S(v)$ the \emph {boundary pattern} of $P(v)$.
Note that two 2-pieces of the same colour are never adjacent on the boundary of a 3-piece.
\vskip 0.1cm
\begin{figure}
\begin{center}
  \includegraphics[scale=0.25]{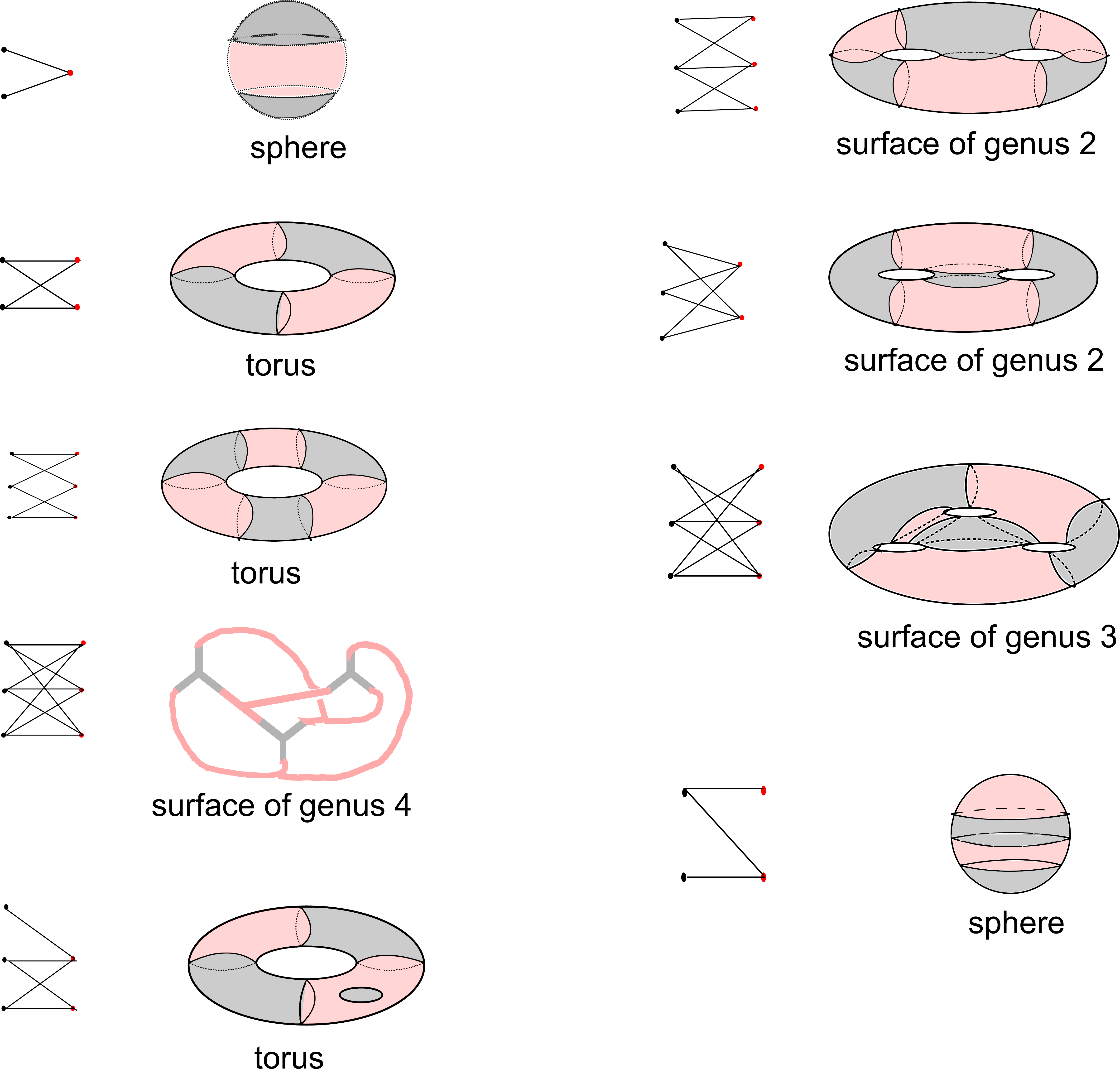}
\end{center}
\caption{How to deduce the boundary pattern $S(v)$ (on the right) given the link of the vertex $v$ (on the left):
if $G$ is the link of the vertex $v$ in $\Delta$, each vertex of valence $k$ in $G$ corresponds to a planar surface having $k$ boundary components (where $k$ is $1$, $2$ or $3$); two surfaces are glued along a circle if and only if the two corresponding vertices are joined by an edge in $G$.
The $3$-piece $P(v)$ is the exterior of the surfaces we draw (if $S(v)$ has positive genus we require that 1-pieces on the boundary are not all trivial in $\pi_1(P(v))$).
} \label{TABLE 3}
\end{figure}

\par With a little abuse we will use the word 2-piece (resp 3-piece) to denote both the open and the closed surface (resp. handlebody).

\vskip 0.5cm
We glue now these pieces together to form a 3-manifold, which we denote as $M_{{\Delta}}$.

The manifold $M_{{\Delta}}$ will be constructed inductively, first taking the union of the 1-pieces, then attaching the 2-pieces and eventually the 3-pieces.
The idea is that we attach an $n$-piece to an $(n-1)$-piece if the corresponding cells of the square complex are incident to each other. The procedure is described below.
\vskip 0.2cm
\par Let $N_s$ be an indexing set for the squares of the complex $\Delta$, $N_e$ be an indexing  set for the
edges of $\Delta$, and $N_v$ be an indexing set for the vertices of $\Delta$.
\vskip 0.1cm
\par Define $C_1$ as the disjoint union of the circles $c(s_i)$ for all $i$ in $N_s$. We call $C_1$ the \emph{1-skeleton} of $M_{\Delta}$.
\vskip 0.1cm
\par Then attach the 2-pieces to the 1-skeleton. Note that, by construction, if $e$ is an edge contained in the square $s$, then exactly one of the boundary components of the 2-piece $p(e)$ will correspond to the square $s$. We attach this boundary component to the 1-piece $c(s)$; the attaching map is meant to be a self-homeomorphism of the circle.
\par \noindent Denote by $C_2$ the set $C_1 \bigsqcup_{i \in N_e} p(e_i)$ quotiented by the attaching maps and endowed with the quotient topology. We choose the attaching maps in such a way that $C_2$ is orientable. Note that such a choice is possible and is unique up to isotopy. We call $C_2$ the \emph{2-skeleton} of $M_{\Delta}$.
By construction, $C_2$ satisfies the following properties:
\par - Each boundary component of a 2-piece is attached to exactly one 1-piece, and two different boundary components of the same 2-piece are attached to two different 1-pieces.
\par - For each 1-piece $c = c(s)$, exactly four different 2-pieces, two black ones and two red ones, are glued to $c$, these are exactly the 2-pieces corresponding to the four edges of the square $s$.
\vskip 0.1cm

Finally, we glue the 3-pieces to the 2-skeleton $C_2$ to form $M_{\Delta}$.
\par Namely, given a vertex $v$ in $\Delta$, consider the subset of $C_2$ composed of 1-pieces and 2-pieces which correspond to edges and squares in the star of $v$.  Note that this is a connected closed orientable surface, and coincides to the boundary pattern of the 3-piece $P(v)$. Therefore we can glue the 3-piece $P(v)$ to its boundary pattern in $C_2$; recall that, if $P(v)$ is not a ball, we glue it in such a way that the 1-pieces in the boundary pattern are not all trivial in $\pi_1(P(v))$.
Here the gluing map is defined only up to Dehn twists around curves on the boundary pattern which are homotopic to 1-pieces (i.e. essential curves in the 2-pieces). For the moment we choose the gluing maps and carry on with the construction. We will observe below (Remark \ref{Dehn twist ok}) that a different choice for the gluing maps would in the end give us a homeomorphic 3-manifold, and therefore our choice is not relevant.

\par Denote by $M_{\Delta}$ the union of $C_2$ and the 3-pieces, quotiented by the attaching maps. Denote by ${Q_B}$ the union of 1-pieces and black 2-pieces, and by ${Q_R}$ the union of 1-pieces and red 2-pieces.
Note that each 2-piece lies on the boundary of exactly two 3-pieces and each 1-piece lies on the boundary of exactly four 3-pieces.
\vskip 0.2cm
In the same way, if $\widetilde{\Delta}$ denotes the universal cover of such a square complex $\Delta$ (equivalently, $\widetilde{\Delta}$ satisfies properties (1)-(5) in Section \ref{The core of two trees}).
we can construct a manifold $M_{\widetilde{\Delta}}$.Denote by $\widetilde{Q_B}$ the union of 1-pieces and black 2-pieces in $M_{\widetilde{\Delta}}$, and by $\widetilde{Q_R}$ the union of 1-pieces and red 2-pieces in $M_{\widetilde{\Delta}}$.

Note that the action of the group $F_g$ on $\widetilde{\Delta}$ induces a free properly discontinuous cocompact action of $F_g$ on $M_{\widetilde{\Delta}}$. We deduce that $M_{\widetilde{\Delta}}$ is a covering space for $M_\Delta$, and $F_g$ is the deck transformation group for the covering map; $\widetilde{Q_R}$ and $\widetilde{Q_B}$ are the full lifts of $Q_R$ and $Q_B$.
We will prove (Lemma \ref{proof simp conn}), that
$M_{\widetilde{\Delta}}$ is in fact the universal cover of $M_{\Delta}$.
\vskip 0.2cm
\par The rest of this section is aimed at proving the following:

\begin{theorem} \label{main claim}
The space $M_\Delta$ is the connected sum of $g$ copies of $S^2 \times S^1$. $Q_R$ and $Q_B$ are two embedded maximal sphere systems in standard form with respect to each other.
\end{theorem}

\par The proof of Theorem \ref{main claim} consists of several steps. First we prove:

\begin{lemma} \label{lemma1}
$M_\Delta$ is a closed topological 3-manifold.
\end{lemma}

\begin{proof}
We claim that each point $q$ in $M_\Delta$ has a neighborhood homeomorphic to $\mathbb{R}^3$, to prove it, we analyse separate cases.

\par The claim is clearly true if the point $q$ belongs to the interior of a 3-piece.
\par  Suppose the point $q$ belongs to the interior of a 2-piece $p$. Now, $p$ lies on the boundary of exactly two 3-pieces: $P_1$ and $P_2$, which are glued along $p$. Therefore there exist a neighborhood $U_1$ of $q$ in $P_1$, homeomorphic to ${\mathbb{R}^3}_-$, and a neighborhood $U_2$ of $q$ in $P_2$, homeomorphic to ${\mathbb{R}^3}_+$, so that $U_1$ and $U_2$ are glued together along their common boundary; their union provides a neighbourhood of $q$ in $M_\Delta$ homeomorphic to $\mathbb{R}^3$.

\par  Finally, suppose that $q$ is contained in a 1-piece $c$. The piece $c$ lies on the boundary of exactly four $2$-pieces and exactly four $3$-pieces.  Again, by choosing suitable neighborhoods of $q$ in
the four $3$-pieces it belongs to, and gluing them together,
we can find a neighborhood of $q$ in $M_{\Delta}$ homeomorphic to $\mathbb{R}^3$.

\par This proves that $M_\Delta$ is a 3-manifold without boundary. Now, $M_\Delta$ is compact because it is a finite union of 3-pieces.
\end{proof}

In the same way we can show that $M_{\widetilde{\Delta}}$ is a topological 3-manifold, though not necessarily compact.
As a next step we prove the following:

\begin{lemma} \label{lemma2}
Each connected component of $Q_B$ or $Q_R$ is an embedded sphere in $M_\Delta$.
\end{lemma}
\begin{proof}

By construction, each 1-piece is an embedded circle in $M_\Delta$ and each 2-piece is an embedded surface; two different 2-pieces are either disjoint or they are glued together along a 1-piece, and each 1-piece bounds exactly two red 2-pieces and two black 2-pieces. Consequently, $Q_B$ and $Q_R$ are embedded surfaces in $M_\Delta$ without boundary, possibly disconnected.
\par Note now that two 2-pieces of the same colour are glued together along a 1-piece if and only if the edges they correspond to are the two horizontal (or vertical) edges of the same square, i.e. there is a bijective correspondence between the hyperplanes perpendicular to black (resp. red) edges and the connected components of $Q_B$ (resp. $Q_R$).

There is indeed  a systematic way to recover a components of $Q_B$ or $Q_R$ from the hyperplane it corresponds to. Namely, if we consider a hyperplane $H$ as a graph embedded in $\mathbb {R} ^3$ then the corresponding surface will be the boundary of a tubular neighborhood of $H$, which is a sphere, since by property 2) $H$ is a finite tree.
\end{proof}

Note that Lemma \ref{lemma2} holds also for $\widetilde{Q_R}$ and $\widetilde{Q_B}$ in $M_{\widetilde{\Delta}}$.

\begin{remark}\label{Dehn twist ok}
As promised, we observe now that, if we had chosen different gluing maps for the 3-pieces, the manifold we obtained would be homeomorphic to $M_{\Delta}$, i.e. $M_{\Delta}$ is well defined. In fact, choosing a different attaching map for a 3-piece would be the same as performing a Dehn surgery of kind $(1,n)$ (longitude preserving) on a tubular neighborhood of a 1-piece.
Lemma \ref{lemma2} implies that a 1-piece bounds an embedded disc in $M_{\Delta}$; therefore such a Dehn surgery does not modify the homeomorphism class of $M_{\Delta}$. The same is true for $M_{\widetilde{\Delta}}$.
\end{remark}

As a next step towards the proof of Theorem \ref{main claim}, we claim that the fundamental group of $M_{\Delta}$ is the free group $F_g$ of rank $g$.
\par We have already observed that the manifold  $M_ {\widetilde{\Delta}}$ is a covering space for $M_\Delta$ and the deck transformation group is the free group $F_g$. We will prove that $M_{\widetilde{\Delta}}$ is simply connected; which implies that
$M_ {\widetilde{\Delta}}$ is the universal cover of $M_\Delta$, hence the claim holds.

The proof of simple connectedness of
$M_ {\widetilde{\Delta}}$ relies on some preliminary lemmas.

\begin{lemma} \label{lemma3}
For each 3-piece $P$ in $M_\Delta$ or $M_{\widetilde{\Delta}}$, its fundamental group $\pi_1(P)$ is supported on the 1-piece components of its boundary pattern;  i. e. there exists a basis for $\pi_1(P)$ so that each generator $\gamma$ is homotopic to a 1-piece in the boundary pattern.

\end{lemma}
\begin{proof}
The proof is a case by case inspection
\end{proof}
\begin{lemma} \label{lemma4}
For each 3-piece $P$ in $M_\Delta$ (resp. $M_{\widetilde{\Delta}}$), the inclusion $P \rightarrow M_ \Delta$ (resp. $P \rightarrow M_{\widetilde{\Delta}}$) induces the trivial map on the level of fundamental groups.
\end{lemma}
\begin{proof}
By lemma \ref{lemma3} each element of $\pi_1(P)$ can be represented as a product of loops each of which is homotopic to a 1-piece.
By lemma \ref{lemma2} each 1-piece lies on a sphere in $M_\Delta$ or $M_{\widetilde{\Delta}}$, therefore
is trivial in $\pi_1(M_\Delta)$ (or $\pi_1 (M_{\widetilde{\Delta}})$).
\end{proof}

The construction we describe next will be a very useful tool in the proof of simple connectedness of $M_{\widetilde{\Delta}}$ and for the proof of Lemma \ref{comp 3-holed spheres}. \label{construction graph G}
\par To fix terminology, we use the term \emph{binary subdivision} of a V-H square complex to define the union of the 1-skeleton and the hyperplanes.
We will show that a graph $G$ isomorphic to the binary subdivision of $\widetilde{\Delta}$ can be embedded into $M_{\widetilde{\Delta}}$,  and that each loop in the graph $G$ is trivial in the fundamental group of $M_{\widetilde{\Delta}}$.
\par To build this graph $G$ in $M_{\widetilde{\Delta}}$, we first build a subgraph $G'$, which is isomorphic to the 1-skeleton of $\widetilde{\Delta}$. Then we add the other vertices and edges.
\par To construct $G'$, take a point $q_P$ in the interior of each $3$-piece of $M_{\widetilde{\Delta}}$; join two points $q_{P_1}$ and $q_{P_2}$ by an edge $\alpha_p$ if the $3$-pieces $P_1$, $P_2$ the points belong to, both contain the 2-piece $p$ on their boundary. We require $\alpha_p$ to be an embedded arc and to intersect only one 2-piece (the 2-piece $p$) in exactly one point. We colour the edges in $G'$ with black or red, each edge inheriting the colour of the $2$-piece it intersects.
We use the term \emph{4-circuits}  to denote loops in $G'$ consisting of the concatenation of four edges. Note that $G'$ is, by construction, isomorphic as a coloured graph to the 1-skeleton of $\widetilde{\Delta}$, and that 4-circuits coincide with loops of minimal length.

To build $G$ from $G'$, we take a new vertex for each  4-circuit in $G'$, and we join it to the midpoints of the four edges composing the circuit.
Namely, consider a  4-circuit in the graph $G$. This circuit corresponds to four 2-pieces $p_1$, $p_2$, $p_3$ and $p_4$ in  $M_{\widetilde{\Delta}}$, all intersecting in a 1-piece $c$.
Denote by $\alpha_i$ the edge in $G'$ dual to the 2-piece $p_i$.
Take a point $q$ in the circle $c$, and for each $i =1...4$ take an arc $\beta_i$ entirely contained in the 2-piece $p_i$ and joining the point $q$ to the arc $\alpha_i$ (see Figure \ref{figure 4}). Call the edges in $G \setminus G'$ \emph{newedges} and colour each newedge with black or red, according to the colour of the 2-piece it lies on.
We will use the word \emph{bisectors} to denote the union of newedges belonging to the same component $\widetilde{Q_R}$ or $\widetilde{Q_B}$. Note that each bisector is an embedded tree in a component of $\widetilde{Q_R}$ or $\widetilde{Q_B}$ and that bisectors correspond exactly to hyperplanes in $\widetilde{\Delta}$.

By construction, each 4-circuit of $G$  (i. e. each loop in $G$ consisting of four edges) is entirely contained in a single 3-piece of $M_{\widetilde{\Delta}}$, and therefore,  by Lemma \ref{lemma4}, is trivial in $\pi_1(M_{\widetilde {\Delta}})$. Consequently, each loop in $G$ (and in particular each loop in $G'$)  is nullhomotopic in  $M_{\widetilde {\Delta}}$
We are now ready to prove the following:

\begin{lemma} \label{proof simp conn}
$M_ {\widetilde{\Delta}}$ is simply connected.
\end{lemma}

\begin{proof}
We have constructed above a graph $G$ embedded in $M_ {\widetilde{\Delta}}$, and we have shown that each circuit in $G$ is trivial in  $\pi_1(M_{\widetilde{\Delta}})$.
To prove the lemma, we show that each loop $l$ in $M_ {\widetilde{\Delta}}$ is homotopic to a loop in the graph $G$.

In fact, let $l$ be a loop in $M_ {\widetilde{\Delta}}$; up to homotopy, we can suppose that $l$ does not intersect any 1-piece, and intersects every 2-piece transversely.
We may as well suppose that $l$ intersects a 2-piece $p$, if at all, in the point $p \cap G'$ (which is a vertex of $G$ by construction). Now, by Lemma \ref{lemma4}, if $P$ is any 3-piece in $M_ {\widetilde{\Delta}}$, then  $l \cap P$ can be homotoped into $G \cap P$.
\end{proof}

\begin{figure}
 \begin{center}
  \includegraphics[scale=0.25]{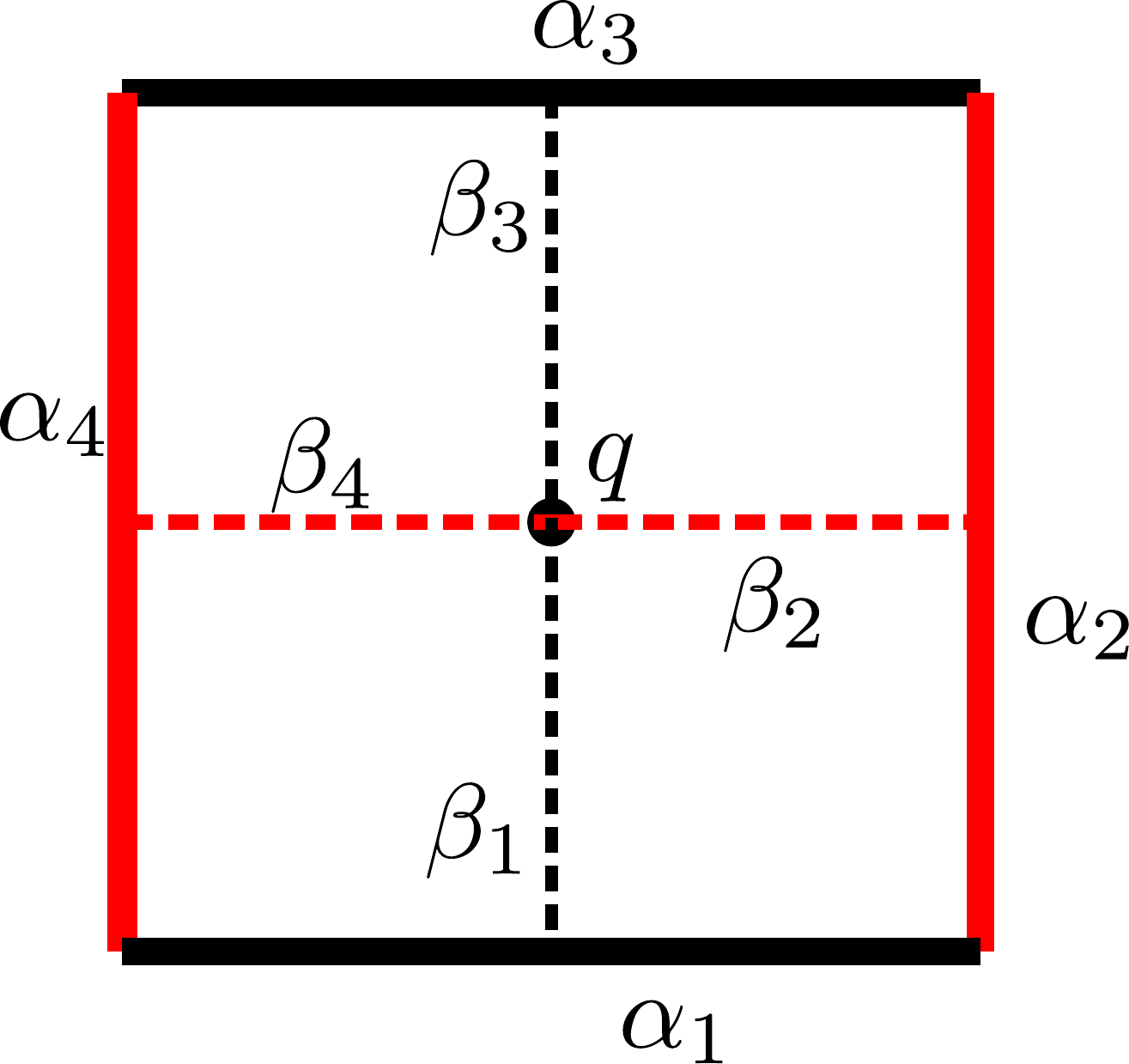}
\caption{A 4-circuit in the graph $G'$ and the newedges. The dotted lines represent the newedges} \label{figure 4}
 \end{center}

\end{figure}

As a consequence of Lemma \ref{proof simp conn}, the fundamental group of $M_{\Delta}$ is the free group $F_g$. Now,
$M_{\Delta}$ is a compact topological 3-manifold without boundary. Using Perelman's solution to Poincar\aca{e} conjecture and 5.3 in \cite{Hem},  we deduce that $M_\Delta$ is the connected sum of $g$ copies of $S^2 \times S^1$.
\vskip 0.3cm
To complete the proof of Theorem \ref{main claim} , we only need to show that $Q_R$ and $Q_B$ are maximal sphere systems in standard form. To reach this goal,  we prove a preliminary Lemma:
\begin{lemma} \label {comp 3-holed spheres}
All complementary components of $\widetilde{Q_B}$ and the complementary components of $\widetilde{Q_R}$ in $M_{\widetilde{\Delta}}$ are three holed 3-spheres.
\end{lemma}
\begin{proof}
\par Let $C$ be a component of $M_{\widetilde{\Delta}}\setminus \widetilde{Q_B}$; note that $C$ is a 3-manifold with boundary and its boundary consists of a certain number of spheres.
Using the construction of the graph $G$ described on page \pageref{construction graph G}, we show that $C$ is simply connected, compact, and has exactly three boundary components. We deduce, using Perelman's solution to Poincar\aca{e} conjecture, that $C$ is a 3-holed 3-sphere.
By symmetry, the same can be proven for any complementary component of $\widetilde{Q_R}$.
\vskip 0.1cm
Note first that $C$ is a union of 3-pieces, 2-pieces and 1-pieces of $M_{\widetilde{\Delta}}$. Here each 1-piece is contained in a boundary component of $C$, hence Lemma \ref{lemma3} implies that for each 3-piece $P$ in $C$ the inclusion $P \rightarrow C$ induces the trivial map on the level of fundamental groups.
Note also that,
if $G \subset   M_{\widetilde{\Delta}}$ is the graph we constructed on page \pageref{construction graph G},
the intersection between a component of $\widetilde{Q_B}$ and the graph $G$ is exactly a black bisector in the graph $G$, hence the intersections between $G$ and the components of $M_{\widetilde{\Delta}} \setminus \widetilde{Q_B}$ are the complementary components of black bisectors in $G$.
Denote by $G_C$ the graph $C \cap G$.
\vskip 0.1cm
Now, each loop in the graph $G_C$ is nullhomotopic in $C$ (since each 4-circuit is entirely contained in a 3-piece).
Moreover, using the same argument as in the proof of Lemma \ref{proof simp conn}, we can show that any loop $l$ in $C$ is homotopic to a loop in the graph $G_C$ (namely, we can suppose that $l$ meets each 2-piece in a vertex of $G_C$, then we observe that for any 3-piece $P$, $l \cap P$ can be homotoped into $G_C \cap P$).
This suffices to deduce that $C$ is simply connected.
\vskip 0.1cm
\par To see that $C$ is compact, note that $G$ is locally finite and bisectors in $G$ are finite  (since hyperplanes in $\widetilde{\Delta}$ are finite); therefore $C$ is the union of a finite number of 3-pieces.
\vskip 0.1cm
To conclude, we show that $C$ has exactly three boundary components.
To this aim, note that black bisectors are perpendicular to black edges, hence, property (5) in Section \ref{The core of two trees} implies that collapsing each red edge and each  black  bisectors in the graph $G$ to a point,  yields a trivalent tree, i. e. each complementary component of black bisectors in $G$ is bounded by exactly three bisectors. As a consequence, $C$ is bounded by exactly three spheres.
\end{proof}

\vskip 0.2cm
We can now prove the following:

\begin{lemma}\label{lemma9}
Each component of $Q_R$ and $Q_B$ is an essential sphere in $M_\Delta$.
Moreover $Q_R$ and $Q_B$ are maximal sphere systems in $M_{\Delta}$, in standard form with respect to each other.
\end{lemma}
\begin{proof}

Denote as usual by $\widetilde {Q_R}$ and $\widetilde {Q_B}$ the full lifts of $Q_R$ and $Q_B$ to the universal cover $M_{\widetilde{\Delta}}$.
By Lemma \ref{comp 3-holed spheres} each component of
$M_{\widetilde{\Delta}} \setminus \widetilde{Q_B}$ and of $M_{\widetilde{\Delta}} \setminus \widetilde{Q_R}$ is a three holed 3-sphere; consequently each component of $M_{\Delta} \setminus Q_B$ and of $M_{\Delta} \setminus Q_R$ is a three holed 3-sphere, which implies that each sphere in $Q_R$ or $Q_B$ is essential and that $Q_R$ and $Q_B$ are maximal sphere systems.
\vskip 0.1cm

\par Now recall that $Q_R$ and $Q_B$ being in standard form means that $Q_R$ and $Q_B$ are in  minimal form (i. e. each sphere in $\widetilde{Q_B}$ intersects each sphere in $\widetilde{Q_R}$ minimally) and that all complementary components of  $Q_R \cup Q_B$ in $M_\Delta$ are handlebodies.
\par The latter condition is satisfied by construction.
\par To see that $Q_R$ and $Q_B$ are in minimal form, note that components of $\widetilde{Q_B}$ and $\widetilde{Q_R}$ correspond to hyperplanes in
$\widetilde{\Delta}$.
Now, $\widetilde{\Delta}$ is  simply connected and locally CAT(0); therefore it is $CAT(0)$, by a generalisation of Cartan-Hadamard Theorem (\citep{BriHae} p. 193). Since two hyperplanes in a CAT(0) cube complex intersect at most once,
 a component of $\widetilde{Q_R}$ and a component of $\widetilde{Q_B}$ intersect at most once in $M_{\widetilde{\Delta}}$. Moreover, by construction, no 3-piece in $M_{\widetilde{\Delta}}$ is bounded by two disks.
These two facts imply that each sphere in $\widetilde{Q_R}$ intersects each sphere in
$\widetilde{Q_B}$ minimally.
\end{proof}
\par This concludes the proof of Theorem \ref{main claim}.

\begin{remark}\label{it is an inverse construction}
Note that the construction we described above is in some sense  \virgolap inverse\virgolch to the one we described in Section  \ref{Dual square complex}, as we explain below.

If we apply the construction described in Section \ref{Dual square complex} to the manifold $M_\Delta$ (resp. $M_{\widetilde{\Delta}}$), then we obtain the complex $\Delta$ (resp. $\widetilde{\Delta}$), i. e. $\Delta$
(resp. $\widetilde{\Delta}$) is the dual square complex associated to $(M_\Delta$, $Q_B$, $Q_R)$ (resp. to $(M_{\widetilde{\Delta}}$, $\widetilde{Q_B}$, $\widetilde{Q_R})$).

On the other hand, if $(\Sigma_1, \Sigma_2)$ is a pair of embedded maximal sphere systems in $M_g$ in standard form, and $\Delta(M_g, \Sigma_1, \Sigma_2)$ is the dual square complex, then applying the above construction to  $\Delta(M_g, \Sigma_1, \Sigma_2)$ yields a 3-manifold $M_\Delta$, with a pair of maximal sphere systems in standard form: $(Q_B, Q_R)$. Then, there exists a homeomorphism $F: M_g \rightarrow M_\Delta$ mapping the pair $(\Sigma_1, \Sigma_2)$ to the pair $(Q_B, Q_R)$. In fact there is a bijective correspondence between the pieces of $(M_g, \Sigma_1, \Sigma_2)$ and the pieces of $(M_\Delta, Q_B, Q_R)$ (since both sets correspond to the cells of $\Delta$), and this correspondence respects the glueing relation (i. e. an $n-1$-piece of $(M_g, \Sigma_1, \Sigma_2)$ lies on the boundary of an $n$-piece of $(M_g, \Sigma_1, \Sigma_2)$ if and only if the same is true for the corresponding pieces of $(M_\Delta, Q_B, Q_R)$). The map $F$ then maps each piece of
$(M_g, \Sigma_1, \Sigma_2)$ to the corresponding piece of $(M_\Delta, Q_B, Q_R)$ homeomorphically.
\end{remark}

\section{Consequences and applications}
\label{Consequences}


A first consequence of the constructions described in Section \ref{Dual square complex} and in Section \ref{Inverse construction} is  that a pair of sphere systems in standard form is somehow determined by its dual square complex. Namely:

\begin{lemma} \label{lemma1consequence2}
Let $(\Sigma_1, \Sigma_2)$, $(\Sigma_1', \Sigma_2')$ be two pairs of embedded maximal sphere systems in $M_g$, both in standard form with respect to each other.
Suppose the dual square complexes $\Delta(M_g, \Sigma_1, \Sigma_2)$ and
$\Delta(M_g, \Sigma_1', \Sigma_2')$ are isomorphic as coloured square complexes. Then there exists a homeomorphism $H: M_g \rightarrow M_g$ so that $H(\Sigma_i)$ is $\Sigma_i'$ for $i=1,2$.
\end{lemma}

\begin{proof}
If we denote by$(M_\Delta, Q_R, Q_B)$
(resp.  $(M_\Delta', Q_R', Q_B')$) the triple obtained by applying the construction of Section \ref{Inverse construction} to
$\Delta(M_g, \Sigma_1, \Sigma_2)$ (resp. to
$\Delta(M_g, \Sigma_1', \Sigma_2')$), then, by construction and by remark \ref{Dehn twist ok}, there is a homeomorphism $M_\Delta \rightarrow M_{\Delta'}$ mapping the pair $(Q_B,  Q_R)$ to the pair
 $(Q_B',  Q_R')$. Hence, Lemma \ref{lemma1consequence2} immediately follows from Remark
\ref{it is an inverse construction}
\end{proof}

\vskip 0.3cm
Note now that, by Proposition \ref{equalityconstructions}, constructing a dual square complex for a triple $(M_g, \Sigma_1, \Sigma_2)$ does not really require $(\Sigma_1, \Sigma_2)$ to be in standard form. In fact, we can construct the dual square complex as (the quotient of) the core of the two dual trees,
endowed with the group actions induced by the  $F_g$-action on $\widetilde{M_g}$.
This observation leads to the following:

\begin{theorem} \label{consequence1}
Let $M_g$ be the connected sum of $g$ copies of $S^2 \times S^1$ and let $\Sigma_1$, $\Sigma_2$ be two embedded maximal sphere systems which satisfy hypothesis ($*$) in the introduction. Then there exist maximal sphere systems $(\Sigma_1 ', \Sigma_2 ')$ such that $\Sigma_i '$ is homotopic to $\Sigma_i$ for $i=1,2$, and $\Sigma_1 '$, $\Sigma_2 '$ are in standard form.
\end{theorem}

Before proving Theorem \ref{consequence1} we clarify some terminology,  given two infinite trivalent trees $T$ and $T'$ endowed with an identification of their boundaries, we say that $T$ and $T'$ \emph{coincide}, if there is a simplicial isomorphism $\varphi : T \rightarrow T'$ such that for each edge $e$ in $T$ its image $\varphi(e)$ induces the same partition as $e$ on the boundary. We are now ready to prove Theorem \ref{consequence1}.

\begin{proof}
Let $\widetilde{\Sigma_1}$ and $\widetilde{\Sigma_2}$ be the full lifts of $\Sigma_1$ and $\Sigma_2$ to $\widetilde{M_g}$ and let $T_1$ and $T_2$ be the dual trees to $\widetilde{\Sigma_1}$ and $\widetilde{\Sigma_2}$ respectively.
Note that $T_1$ and $T_2$ are trivalent trees and they are both endowed with a geometric action by the group $F_g$, induced by the action of $F_g$ on $\widetilde{M_g}$.
Let $C(T_1, T_2)$ be the core of $T_1$ and $T_2$. 
By applying the construction of Section  \ref{Inverse construction} to  $C(T_1, T_2)$, we obtain a triple $(M_C, \widetilde{Q_B}, \widetilde{Q_R})$, endowed with a geometric action of the group $F_g$, inherited by the group action on $C(T_1, T_2)$.  Here $M_C$ is homeomorphic to $\widetilde{M_g}$ and
$\widetilde{Q_B}$,  $\widetilde {Q_R}$ are two embedded maximal sphere systems in standard form with respect to each other.

\par  By construction, $C(T_1, T_2)$ is the  dual square complex to the triple  $({M_C}, \widetilde{Q_B}, \widetilde{Q_R})$, and $T_1$, $T_2$ coincide with the dual trees to $\widetilde{Q_B}$
and $\widetilde{Q_R}$ respectively.
\par The space of ends of $M_C$ can be identified to the space of ends of $\widetilde{M_g}$, as they both can be identified to the boundaries of $T_1$ and $T_2$. Moreover, since the tree dual to $\widetilde{M_g}$ and $\widetilde{\Sigma_1}$ coincides to the tree dual to $M_C$ and $\widetilde{Q_B}$ (they both coincide with the tree $T_1$), then for each sphere $\sigma$ in $\widetilde{\Sigma_1}$ there is a sphere in $\widetilde{Q_B}$ inducing the same partition as $\sigma$ on the space of ends, and for each sphere $s$ in $\widetilde{Q_B}$ there is a sphere in $\widetilde{\Sigma_1}$ inducing the same partition as $s$.
The same holds for $\widetilde{\Sigma_2}$ and $\widetilde{Q_R}$.
\par We can find an $F_g$-equivariant homeomorphism $H: M_C
\rightarrow \widetilde{M_g}$ which is consistent with the identification on the space of ends (inherited from identifying both spaces of ends with the boundaries of $T_1$ and $T_2$).
Denote $H(\widetilde{Q_B})$ by $\widetilde{\Sigma_1'}$ and  $H(\widetilde{Q_R})$ by $\widetilde{\Sigma_2'}$.
\par Now, the systems $\widetilde{\Sigma_1'}$ and
$\widetilde{\Sigma_2'}$ are maximal and are in standard form with respect to each other, since they are homeomorphic image of two maximal sphere systems in standard form.
Moreover, for each sphere in $\widetilde{\Sigma_1}$ (resp. $\widetilde{\Sigma_2}$) there is a sphere in $\widetilde{\Sigma_1'}$ (resp. $\widetilde{\Sigma_2'}$) inducing the same partition on the space of ends and vice versa.
Hence, by Lemma \ref{spherepartition} for $i=1,2$ the sphere system  $\widetilde{\Sigma_i}$ is homotopic in $\widetilde{M_g}$ to the sphere system  $\widetilde{\Sigma_i'}$.
\par  Let $\Sigma_1'$ and  $\Sigma_2'$ in $M_g$ be the projections of
$\widetilde{\Sigma_1'}$ and $\widetilde{\Sigma_2'}$ through the covering map. These are two embedded maximal sphere systems in $M_g$ in standard form with respect to each other, and moreover for $i=1,2$, the sphere system $\Sigma_i '$ is homotopic in $M_g$ to the sphere system
$\Sigma_i$. 
\end{proof}

\par To summarise, in the proof of Theorem \ref{consequence1}, we have shown a constructive way to find a standard form for two maximal sphere systems in $M_g$.
Note that Theorem \ref{consequence1} can also be proven using the existence of Hatcher's normal form (Proposition 1.1 in \citep{Hat1}).


\vskip 0.2cm
Another consequence of the construction of Section \ref{Inverse construction} is a kind of uniqueness result for standard form. More precisely:

\begin{theorem} \label{consequence2}
Let $(\Sigma_1, \Sigma_2)$, $(\Sigma_1 ', \Sigma_2 ')$ be two pairs of embedded maximal sphere systems in $M_g$. Suppose that both pairs of sphere systems are in standard form and satisfy hypothesis ($*$) in the introduction.
Suppose also that $\Sigma_i$ is homotopic to $\Sigma_i '$ for $i=1,2$.
Then there exists a homeomorphism $F:M_g \rightarrow M_g$
such that $F(\Sigma_i)= \Sigma_i '$ for $i=1,2$.
The homeomorphism $F$ induces an inner automorphism of the fundamental group of $M_g$.
\end{theorem}

The proof of Theorem \ref{consequence2} is based on Lemma \ref{lemma1consequence2} and on the following result.

\begin{lemma} \label{lemma2consequence2}
For $g \geq 3$, let $F:M_g \rightarrow M_g$ be a self-homeomorphism of $M_g$. Let $\Sigma$ be a maximal sphere system  in $M_g$. Suppose that for each sphere $\sigma$ in $\Sigma$ the image $F(\sigma)$ is homotopic to $\sigma$.
Then the induced homomorphism $F_* :\pi_1(M_g)
\rightarrow \pi_1(M_g)$ is an inner automorphism of the free group $F_g$.
\end{lemma}
Lemma \ref{lemma2consequence2} is well known, however, since we did not find a reference, we give a proof below.

\begin{proof}
Denote as usual by $\widetilde{M_g}$ the universal cover of $M_g$, and denote the full lift of $\Sigma$ by $\widetilde{\Sigma}$. The manifold $\widetilde{M_g}$ is endowed with an action by the free group $F_g$ and the quotient of  $\widetilde{M_g}$ by this action is the manifold $M_g$. In order to prove Lemma \ref{lemma2consequence2} we will show that a lift $\widetilde{F}$ of the homeomorphism $F$ is equivariant under this group action.
\par To this aim, first note that a homeomorphism $H: \widetilde{M_g} \rightarrow \widetilde{M_g}$ induces a homeomorphism $H_E: End(\widetilde{M_g}) \rightarrow End(\widetilde{M_g})$. Note also
that the $F_g$-action on $\widetilde{M_g}$ induces an action of $F_g$ on the space of ends; on the other hand, this action on the space of ends determines the action on  $\widetilde{M_g}$ up to homotopy (in fact, since each component of $\widetilde{M_g}\setminus \widetilde{\Sigma}$ is a 3-holed 3-sphere, then the action of $F_g$ on $\widetilde{M_g}$ is determined by the action of $F_g$ on $\widetilde{\Sigma}$; and the action of $F_g$ on $\widetilde{\Sigma}$ is determined up to homotopy by the action of $F_g$ on the space of ends of $\widetilde{M_g}$). Consequently, a homeomorphism $H: \widetilde{M_g} \rightarrow \widetilde{M_g}$ is $F_g$-equivariant (up to homotopy) if and only if the induced map $H_E: End(\widetilde{M_g}) \rightarrow End(\widetilde{M_g})$ is equivariant under the induced $F_g$-action on $End(\widetilde{M_g})$.
Note also that, since each component of $\widetilde{M_g} \setminus \widetilde{\Sigma}$ is a 3-holed 3-sphere, then a map $H: \widetilde{M_g} \rightarrow \widetilde{M_g}$ is determined, up to homotopy, by its behaviour on the spheres in $\widetilde{\Sigma}$.
\par Now, let $\widetilde{\sigma}$ be a sphere in $\widetilde{\Sigma}$. Since $F$ fixes the homotopy class of each sphere in $\Sigma$, then we can choose $\widetilde{F}: \widetilde{M_g} \rightarrow \widetilde{M_g}$ in such a way that the sphere $\widetilde{F}(\widetilde{\sigma})$ is homotopic to the sphere $\widetilde{\sigma}$ in $\widetilde{M_g}$.
In addition, the image $\widetilde{F}(\widetilde{\sigma})$ determines the image
$\widetilde{F}(\widetilde{\tau})$ for each $\tau$ in $\widetilde{\sigma}$ (here we are using that, since $g \geq 3$, a triple of spheres in ${\Sigma}$ bounds at most one  component of $M_g \setminus \Sigma$). This means that $\widetilde{F}$ fixes the homotopy class of each sphere in $\widetilde{\Sigma}$. Hence, for each $\widetilde{\tau}$ in $\widetilde{\Sigma}$, the sphere $\widetilde{F}(\widetilde{\tau})$ induces the same partition as the sphere $\widetilde{\tau}$ on the space of ends of  $\widetilde{M_g}$.
Consequently, the homeomorphism  $\widetilde{F}_E :End(\widetilde{M_g}) \rightarrow End(\widetilde{M_g})$ induced by $\widetilde{F}$ is equivariant under the $F_g$-action on $End(\widetilde{M_g})$, which implies that $\widetilde{F}$ is equivariant, up to homotopy, under the $F_g$-action on $\widetilde{M_g}$.
\end{proof}

\begin{remark} \label{case g=2}
Lemma \ref{lemma2consequence2} holds true also in the case where $g$ is two, under the additional hypothesis that $F$ fixes the components of $M_g \setminus \Sigma$ up to homotopy; in particular it holds true when $F$ is orientation preserving.
\end{remark}

We go on to prove Theorem \ref{consequence2}.

\begin{proof} (of Theorem \ref{consequence2})
Let $T_1$, $T_2$, $T_1'$, $T_2'$ be the dual trees to $\widetilde{\Sigma_1}$, $\widetilde{\Sigma_2}$, $\widetilde{\Sigma_1'}$ and $\widetilde{\Sigma_2'}$ respectively. Since, for $i=1,2$, the system $\Sigma_i$ is homotopic to the system  $\Sigma_i'$, then the core $C(T_1, T_2)$ is isomorphic as a V-H square complex to the core $C(T_1', T_2')$, and the quotients $\Delta(T_1, T_2)$ and $\Delta(T_1', T_2')$ are also isomorphic.
Thus, by Proposition \ref{equalityconstructions}, the square complex dual to $(M_g, \Sigma_1, \Sigma_2)$ is isomorphic, as a coloured square complex, to the square complex
dual to $(M_g, \Sigma_1', \Sigma_2')$. 
By Lemma \ref{lemma1consequence2}, there exists a homeomorphism $F: M_g \rightarrow M_g$ such that, for $i=1,2$, the image $F(\Sigma_i)$ is $\Sigma_i'$. Indeed, since $F$ is induced by the isomorphism of square complexes, $F$ fixes the homotopy class of each sphere in $\Sigma_1$ and $\Sigma_2$, and of each component of $M_g \setminus \Sigma_1$.
Hence, by Lemma \ref{lemma2consequence2} (and by remark \ref{case g=2} if $g=2$), the map $F$ induces an inner automorphism of the fundamental group $F_g$.
\end{proof}
We conclude this section with the following:

\begin{remark}
A theorem by Laudenbach (\citep{Lau2} page 80) states that if $Mod(M_g)$ denotes the group of (isotopy classes of) self-homeomorphisms of the manifold $M_g$ and $H: Mod(M_g) \rightarrow Out(F_g)$ is the homomorphism sending a map to its action on $\pi_1(M_g)$, then the kernel of this map is the  subgroup of $Map(M_g)$ generated by a finite number of sphere twists (namely twists around spheres in a maximal sphere system).
In light of this result, we can restate  Theorem \ref{consequence2} in the following way:
\vskip 0.1cm
\par Statement:  Two standard forms for a pair of maximal sphere systems, $(\Sigma_1, \Sigma_2)$ differ by a combination of sphere twists in the manifold $M_g$.




\end{remark}

\bibliographystyle{plainnat}
\bibliography{sample}
\addcontentsline{toc}{chapter}{Bibliography}

\end{document}